\documentclass[a4paper]{amsart}

\usepackage{color}
\usepackage{amsxtra}
\usepackage{mathtools}
\usepackage{enumerate}
\usepackage[english]{babel} 
\usepackage{blindtext}
\usepackage[inline]{enumitem}
\usepackage{nicefrac}
\mathtoolsset{showonlyrefs} 
\usepackage{ mathrsfs }
\usepackage{ upgreek }
\usepackage{tikz}

\usetikzlibrary{shadings}
\usepackage{color}

\usepackage{hyperref}
\hypersetup{
  colorlinks   = true, 
  urlcolor     = blue, 
  linkcolor    = blue, 
  citecolor   = red 
}
\xdefinecolor{darkgreen}{RGB}{175, 193, 36}

\newtheorem{theorem}{Theorem}[section]
\newtheorem{lemma}{Lemma}[section]

\newtheorem{definition}{Definition}[section]

\newtheorem{remark}{Remark}[section]
\DeclareRobustCommand{\rchi}{{\mathpalette\irchi\relax}}
\newcommand{\irchi}[2]{\raisebox{\depth}{$#1\chi$}} 
\title[Fundamental solution]{The fundamental solution of the fractional $p-$laplacian}

\author[L. M. Del Pezzo]{Leandro M. Del Pezzo}
	\address{Leandro M. Del Pezzo \hfill\break\indent
		IESTA -- Dpto. de Métodos Cuantitativos\hfill\break\indent
		Facultad de Ciencias Económicas y de Administración  \hfill\break\indent
		Universidad de la República
		\hfill\break\indent Av. Gonzalo Ramírez 1926, 11200 Montevideo, 
		\hfill\break\indent Departamento de Montevideo, URUGUAY. }
	\email{leandro.delpezzo@fcea.edu.uy}
	\urladdr{http://cms.dm.uba.ar/Members/ldpezzo/}
\author[A. Quaas]{Alexander Quaas}
\address{A. Quaas
        \hfill\break\indent Departamento de Matem\'atica, 
        \hfill\break\indent Universidad T\'ecnica Federico Santa Mar\'ia 
        \hfill\break\indent Casilla V-110, Avda. 
        \hfill\break\indent Espa\~na, 1680 -- Valpara\'iso, CHILE.}
    \email{alexander.quaas@usm.cl}
\begin{document}


\begin{abstract}
	In this article, we find the fundamental solution of the fractional
	$p-$laplacian and use them to prove two different Liouville--type theorems.
	 A non-existence classical Liouville-type theorem for  $p$-superharmonic 
	and a Louville type results for an Emden-Folder type equation with the fractional $p-$laplacian.
\end{abstract}
\maketitle

\keywordsname{ Fundamental solution; Liouville-type theorems; fractional
			$p-$laplacian; non-local operator.}



\section{Introduction}

	In the last decades, Liouville-type non-existence theorems have been studied intensively, 
	because they have emerged as a crucial tool for many applications in PDEs. They mostly appear in 
	establishing qualitative properties of solutions. The best known is the Gidas-Spruck a priori bound and nowadays 
	Liouville-type theorems are used in regularity issues.
	Observe that the non-existence results are used, in most of the cases, after rescaling and 
	a compactness argument. See, for instance, 
	\cite{MR4294719, MR3967048,MR2739791,MR1990294} and references therein.
	
	\medskip

	Our purpose here is to establish Liouville-type theorems for equations that involve the fractional 
	$p$-laplacian, defined by, 
	\[
		(-\Delta_p)^su(x)\coloneqq
		2C(s,p,N)
		\lim_{\varepsilon\to0^+}
		\int_{\mathbb{R}^N\setminus B_\varepsilon(x)}\!\!\!\!\!\!\!\!\!\!\!\!\!\!
		\!\!\!\!\dfrac{\Psi_p(u(x)-u(y))}{|x-y|^{N+sp}}dy\quad
		x\in\mathbb{R}^N
	\]
	where $p\in(1,\infty),$ $s\in(0,1),$ and 
	$C(s,p,N)$ is a normalization factor.  
	This constant ensures that in the limits
	\(p\to 2\) and \(s\to 1,\)  the operator coincides with the standard fractional laplacian
	and the classical \(p-\)laplacian, respectively. 
	In this work, for simplicity, we omit this constant, as it does not affect the qualitative properties of the solutions we study.
	
	\medskip
	
	The non-local operators have become relevant because they arise in several applications in many 
	fields, for instance,  game theory, mathematical physics,
	finance, image processing, L\'evy processes in probability, and 
	some optimization problems, 
	see  \cite{Caffarelli2012,MR2042661,MR2480109,MR1406564,MR0521262}
	and the references therein. From a mathematical point of view,
	the fractional $p-$Laplacian has a great 
	interest since it exhibits two key features: the
	nonlinearity of the operator and its non-local character. 
	For instance, the literature includes works on global bifurcation \cite{MR3556755}, eigenvalues \cite{MR3552458,MR3411543,LL}, regularity 
	\cite{MR3542614,MR3593528,MR4109087}, evolution problems \cite{MR3491533,MR3456825}, and existence results via Moser iteration \cite{MR3483598}, each contributing key insights into different aspects of equations involving the fractional \(p-\)laplacian.

\subsection{Main results}
	Our first result introduces the general formula for a fundamental solution of the fractional 
	$p-$laplacian. 
	\begin{theorem}\label{Theorem:Fundamental}
		Let $N\ge 2,$ $0<s<1,$ and  $1<p<\infty.$  
		\begin{enumerate}
			\item[(a)] If $ps\neq N$ then 
				\[
					v_\beta(x)=|x|^{\beta}\quad \beta\in\left(-\tfrac{N}{p-1},\tfrac{ps}{p-1}\right),
				\] 
				is a weak solution of 
				\begin{equation}
					\label{eq:psneqNteo}
						(-\Delta_p)^s v_\beta(x)= 
							\mathcal{C}(\beta)|x|^{\beta(p-1)-sp}\quad\text{in }\mathbb{R}^N\setminus\{0\},
				\end{equation}
				where
				\begin{equation}\label{eq:cbeta}
					\mathcal{C}(\beta)\coloneqq
					4\pi\alpha_N\int\displaylimits_{0}^{1}
					\left|1-\rho^{\beta} \right|^{p-2}
					\left(1-\rho^{\beta}\right)\left[\rho^{N-1}
					-\rho^{ps-\beta(p-1)-1}\right]G(\rho^{2})d\rho,
				\end{equation}
				with
				\begin{align*}
					&	\alpha_N\coloneqq\dfrac{\pi^{\frac{N-3}2}}{\Gamma\left(\frac{N-1}{2}\right)} 
					\text{ and }\\
					&G(t)\coloneqq G(t,N,ps)\coloneqq 
					B\left(\frac{N-1}2,\frac{1}2\right)F\left(\frac{N+ps}2,\frac{ps+2}2;\frac{N}2;t\right).
				\end{align*}
				Here $\Gamma,B$ and $F$ denote the gamma, 
				the beta, and the  (2-1)-hypergeometric functions, 
				respectively.
				
				Additionally, we have
				\begin{equation}\label{signodecbeta}
					\mathcal{C}(\beta)
					\begin{cases}
						=0 &\text{if } \beta=0,\text{ or }\beta=\tfrac{ps-N}{p-1},\\
						>0 &\text{if } \min\{\tfrac{ps-N}{p-1},0\}<\beta<\max\{\tfrac{ps-N}{p-1},0\},\\
						<0 & \text{otherwise}.
					\end{cases}
				\end{equation}
			\item[(b)] If $ps= N$ then 
				\[
					v(x)=\log(|x|),
				\] 
				is a weak solution of 
				\begin{equation}\label{eq:ps=Nteo}
					(-\Delta_p)^s v(x)= 0\quad\text{in }\mathbb{R}^N\setminus\{0\}.
				\end{equation}
		\end{enumerate}

	\end{theorem}
	
	\medskip

	It should be noted that point \textit{(a)} of the above theorem is mentioned in \cite[Example 1.5]{MR3861716}, but this is 
	only in the case where $2 < p \le N+1$ and $0 < s < \frac{p-1}{p}$, and without a detailed proof. 
	Additionally, in \cite[Appendix A]{MR3461371}, the result is proven for the case $N > sp$ and $\frac{N-sp}{p} \le \beta < \frac{N}{p-1}$. 
	Here, we provide a complete proof covering all cases, which requires delicate estimates along with some nontrivial explicit computations.
	
	We also observe that a similar result holds for $N=1$, though in this work we focus on the case $N \ge 2$.
	
	The assertion of Theorem \ref{theorem:Liouville} is formally natural due to the scaling properties of the fractional integral operator. 
	Indeed, when considering a rescaling, $u_\lambda(x) = u(\lambda x)$, the fractional integral scales accordingly, 
	suggesting that the result should hold at least in a heuristic sense. However, making this argument rigorous presents substantial 
	challenges due to regularity issues arising from the nonlocal nature of the operator. Establishing the result requires careful analysis, particularly in the 
	delicate logarithmic case. Moreover, determining the precise sign of the constant $\mathcal{C}(\beta)$ 
	in terms of $\beta$ is itself a subtle and nontrivial task.

 	\begin{remark}
	    By \cite{MR4333435}, \(v_\beta\) and \(v\) solutions are viscosity solutions. 
	    Furthermore, since our solutions are  $C^\infty(\mathbb{R}^{N}\setminus\{0\})$ 
	    viscosity solutions that have non-zero gradients in $\mathbb{R}^{N}\setminus\{0\}$,
	    they are also classical solutions.
		For the definitions of viscosity and weak solutions, see 
		Section \ref{laprevia}.
	\end{remark}	
	
	\begin{remark}
		A review of fundamental solutions in the local case even for fully nonlinear operators can be found in \cite{MR2663711}. In this scenario, distributional setting is not possible so the equations hold only in $\mathbb{R}^{N}\setminus\{0\}$, this is our approach. In the classical literature fundamental solution satisfies also the equation in $\mathbb{R}^{N}$ with $\delta_0$ as a right hand side. In our setting when $ps-N\not=0$ the fundamental solution corresponds to the case $\beta=\tfrac{ps-N}{p-1}.$ for the above power function and in the standard case (that is $p=2$ and $s=1$), we have $\beta=2-N$. Here, we abuse of the fundamental solution name since we include all values of the power $\beta$ in the corresponding admissible range. 						
	\end{remark}

	In a parabolic context,  the fundamental solution for the fractional $p-$laplacian
	in self-similar variables is found in \cite{MR4114983}.
	
	\medskip

    As mentioned before, our main application of the fundamental solution of the fractional
	$p-$laplacian is to establish two models of Liouville-type theorems depending on the order between  $N$ and $ps$. We start with the case $N\le ps$ that corresponds to 
	the classical type of Liouville results.

	\begin{theorem}[First Liouville-type theorem]\label{theorem:Liouville} 
		Let $N\ge2,0<s<1,$ and $1<p<\infty.$ If $N\le ps$ and $u$ is a non-negative
		lower semi-continuous weak solution of
		\begin{equation}\label{eq:supersol}
			(-\Delta_p)^s u \ge 0\quad\text{in }\mathbb{R}^N,
		\end{equation}
		then $u$ is constant.
	\end{theorem}
	
	\begin{remark}\label{remark.visco.weak}
		Again as the previous remark by \cite{MR4030247}, we know that a non-negative lower semi-continuous 
		weak solution of \eqref{eq:supersol} is also a viscosity solution. 
		Then the previous theorem also holds if we consider viscosity solutions 
		instead of weak ones. See also \cite{MR4333435}.
	\end{remark}

	Our second Liouville-type theorem is for an equation involving the fractional $p-$laplacian operator
	(with $N>ps$) and a zero-order power nonlinearity, this corresponds to in the literature as the Lane-Emden type equations.
	
	\begin{theorem}[Second Liouville-type theorem]
	\label{theorem:nonlinear}
		Let $N\ge2,0<s<1,$ $1<p<\infty,$ and $N>ps.$  
		\begin{itemize}
			\item If $0<q< \tfrac{N(p-1)}{N-ps}$ and 
			$u\in C(\mathbb{R}^N)$ is a non-negative 
			 viscosity solution of 
			\begin{equation}\label{eq:nl1}
				(-\Delta_p)^s u-u^q \ge 0\quad\text{in }\mathbb{R}^N
			\end{equation}
			then $u\equiv0.$
            		  \item If $q>\tfrac{N(p-1)}{N-ps}$ then there is a positive 
            		  solution of 
		          \eqref{eq:nl1}.
		\end{itemize}

	\end{theorem}

	For the proofs of our Liouville-type theorems, we proceed similarly to \cite{MR2739791},  but we need  some extra
	delicate estimates together with some new ideas due to the strongly nonlinear character of the operator.
	A review on Liouville-type theorems of this type can be found in \cite{MR2739791},  for other previous results see also 
	\cite{MR4149690,MR1321809, MR0577096,MR1723262}.
	
	\medskip
	
	For the $p$-laplacian (case $s=1$), the last 
	two theorems were proved in \cite{MR1990294}, see also \cite{MR1723262} where even systems are consider. Furthermore, in \cite{MR1990294}, it is shown that the first point of Theorem \ref{theorem:nonlinear} also holds when 
	$q= \tfrac{N(p-1)}{N-ps}$, sometimes known as  the critical case for super-solution. 
	Unfortunately, we have not yet been able to establish this result within our current framework, so we leave it as an open problem, stated below. The difficulty in our approach is to compute a log perturbation of the fundamental solution thus 
	is related with a product rule for the fractional $p-$laplacian.
	
	\medskip
	
	\noindent{\bf Open problem:} {\it Let $N\ge2,0<s<1,$ $1<p<\infty,$ and $N>ps.$
		If $q=\tfrac{N(p-1)}{N-p}$ and 
			$u\in C(\mathbb{R}^N)$ is a non-negative 
			 viscosity solution of \eqref{eq:nl1} then $u\equiv0.$
	} 
	
	\medskip
	
	In this work, we focus on dimensions \(N\ge 2\) and exclude the one-dimensional case \(N=1.\) We note, however, that the case 
	\(N=1\) might offer additional insight into the behavior of solutions, potentially simplifying certain arguments or even allowing 
	the construction of explicit counterexamples to open questions, such as higher regularity. 
	This case could be an interesting direction for future investigation.

	\subsection*{The paper is organized as follows.} In Section \ref{laprevia}, 
	we give the definition of weak and viscosity solutions. 
	In Section \ref{sfs}, we prove Theorem \ref{Theorem:Fundamental}.
	 Afterward, in Section \ref{hp}, we prove some Hadamard properties that will be fundamental 
	 to proving our Liouville results. Finally, in Sections \ref{st1}  and \ref{st2}, we prove Theorems
	 \ref{theorem:Liouville} and \ref{theorem:nonlinear}.


\section{preliminaries}\label{laprevia}
		Throughout this paper, $\Omega$ is an open set of 
		$\mathbb{R}^N,$  and
		$s\in(0,1),$ $p\in(1,\infty).$ 	
		The fractional Sobolev spaces $W^{s,p}(\Omega)$
		is defined to be the set of functions $u\in L^{p}(\Omega)$ such that
		\[
			|u|_{W^{s,p}(\Omega)}^p\coloneqq
			\int_{\Omega^2}
			\dfrac{|u(x)-u(y)|^p}{|x-y|^{N+sp}}\, dxdy<\infty,
		\]
		where $\Omega^2$ denotes $\Omega\times\Omega.$
		The fractional Sobolev spaces admit the following norm
		\[
			\|u\|_{W^{s,p}(\Omega)}\coloneqq\left(\|u\|_{L^{p}(\Omega)}^p+|u|_{W^{s,p}(\Omega)}^p
			\right)^{\frac1p},
		\]
		where
		\[
			\|u\|_{L^{p}(\Omega)}^p\coloneqq\int_\Omega |u(x)|^p\, dx.
		\]
		We also denote 
		\[
			L_{s}^{p-1}(\mathbb{R}^N)\coloneqq
				\left\{
					u\in L^{p-1}_{loc}(\mathbb{R}^N)\colon
					\int_{\mathbb{R}^{N}}
						\frac{|u|^{p-1}}{(1+|x|)^{N+sp}}dx
						<\infty
				\right\}.
		\]
		
		For convenience, we will adopt the notation \(\Psi_p(t) = |t|^{p-2} t\) throughout this paper.
		
		\begin{remark}
			Let us note that, in Theorem \ref{Theorem:Fundamental}, 
			we chose 
			$\beta\in\left(-\tfrac{N}{p-1},\tfrac{ps}{p-1}\right)$ 
			so that
			$v_\beta(x)=|x|^{\beta}\in L_{s}^{p-1}(\mathbb{R}^N).$ 
		\end{remark}
		
		\begin{definition}[Weak solution]
			Let $f\colon\Omega\times\mathbb{R}\to\mathbb{R}$ 
			be a continuous function.
			A function $u\in L_{s}^{p-1}(\mathbb{R}^N)$ is 
			a weak super-solution (sub-solution) of 
			\begin{equation}\label{eq:weak}
				(-\Delta_p)^su(x)=f(x,u) \text{ in } \Omega,
			\end{equation}
			if for any bounded open $U\subseteq\Omega$ 
			we have that $u\in W^{s,p}_{loc}(U)$ and
			\[
				\int_{\mathbb{R}^{2N}}
				\frac{\Psi_p(u(x)-u(y))(\varphi(x)-
				\varphi(y))}{|x-y|^{N+sp}}dxdy
				\ge(\le)\int_{\mathbb{R}^{N}}f(x,u)\varphi(x)dx,
			\]
			for any non-negative function 
			$\varphi\in C^\infty_c(U).$ 
			We say that $u$ is a weak solution of 
			\eqref{eq:weak} if it is both a weak super-solution and sub-solution to the problem.
		\end{definition}
		
		\medskip
		
		Following \cite{MR4030247}, we define our notion of
		viscosity super-solution of \eqref{eq:weak}. We start to introduce some notation.
		
		\medskip 
		
		The set of critical points
		of a differentiable function $u$ and the distance  from the
		critical points are denoted by
		\[
			N_u\coloneqq\{x\in\Omega\colon\nabla u(x)=0\},
			\quad\text{and}
			\quad d_u(x)\coloneqq\mathrm{dist}(x,N_u),
		\]
		respectively. Let $D\subset\Omega$ be an open set. We denote
		the class of $C^2-$functions whose gradient and Hessian are
		controlled by $d_u$ as
		\[
			C^2_{\gamma}(D)\coloneqq\left\{
			u\in C^2(\Omega)\colon \sup_{x\in D}
			\left(\dfrac{\min\{d_u(x),1\}^{\gamma-1}}
			{|\nabla u(x)|}+\dfrac{|D^2u(x)|}{d_u(x)^{\gamma-2}}
			\right)<\infty\right\}.
		\]
		
		Observe that, if $\gamma\ge 2 $ then  $u(x)=|x|^\gamma\in C^2_\gamma.$

		\medskip
		
		\begin{definition}[Viscosity Solution]
			Let $f\colon\Omega\times\mathbb{R}\to\mathbb{R}$ be a continuous function.
			We say that a function $u\colon\mathbb{R}^N\to[-\infty,\infty]$
			is a viscosity super-solution (sub-solution) of \eqref{eq:weak} if it satisfies
			the following four assumptions:
			\begin{enumerate}
				\item[(VS1)] $u<\infty$ ($u>-\infty$) a.e. in $\mathbb{R}^N$ and
					$u>-\infty$ ($u<\infty$) everywhere in $\Omega;$
				\item[(VS2)] $u$ is lower (upper) semi-continuous in $\Omega;$
				\item[(VS3)] If $\phi\in C^2(B_r(x_0))$ for some 
				$B_r(x_0)\subset\Omega$ such that $\phi(x_0)=u(x_0)$ and 
				$\phi\le u$ ($\phi\ge u$) in $B_r(x_0),$ and one of the following holds
				\begin{enumerate}
					\item[(a)]$p>\tfrac2{2-s}$ or 
						$\nabla \phi(x_0)\neq0;$
					\item[(b)]$1<p\le\tfrac2{2-s};$
						$\nabla \phi(x_0)=0$ such that $x_0$ is an isolate
						critical point of $\phi,$ and 
						$\phi\in C^2_{\gamma}(B_r(x_0))$ 
						for some
						$\gamma>\tfrac{sp}{p-1};$ 
				\end{enumerate}
				then $(-\Delta_p)^s\phi_r(x_0)\ge (\le) f(x_0,u),$ where
				\begin{equation}\label{ec:fir}
					\phi_r(x)=
					\begin{cases}
						\phi &\text{ if } x\in B_r(x_0),\\
						u(x) &\text{otherwise};
					\end{cases}
				\end{equation}
				\item[(VS4)] $u_-\coloneqq\max\{-u,0\}$ $(u_+\coloneqq\max\{u,0\})$
					belongs to $\in L_{s}^{p-1}(\mathbb{R}^N).$   
			\end{enumerate}
			Finally, $u$ is a 
			viscosity solution if it is both a viscosity
			super-solution and sub-solutions.
		\end{definition}


\section{Fundamental solution}\label{sfs}
	In this section, we will prove the main result of this article (Theorem \ref{Theorem:Fundamental}).
	To simplify the presentation, we split the proof into two cases. 

\subsection{Case \texorpdfstring{$ps\neq N$}{}}
	
	This subsection aims to prove the following result.
	
	\begin{theorem} \label{theorem:casepsneqN}
		Let $N\ge2,0<s<1,$ and $1<p<\infty.$ If $ps\neq N$ then 
		\[
			v_\beta(x)=|x|^{\beta}\quad \beta\in\left(-\tfrac{N}{p-1},\tfrac{ps}{p-1}\right),
		\] 
		is a weak solution of 
		\begin{equation}\label{eq:psneqN}
			(-\Delta_p)^s v_\beta(x)= \mathcal{C}(\beta)|x|^{\beta(p-1)-sp}\quad\text{in }\mathbb{R}^N\setminus\{0\},
		\end{equation}
		where $\mathcal{C}(\beta)$ is defined by \eqref{eq:cbeta}.
	\end{theorem}
		
	
    We introduce a bit of notation to take advantage of the fact that $v_\beta$ is a radial function.
    Given $\varepsilon>0$ and  we define 
	\[ 
		A_\varepsilon(x)\coloneqq\{y\in\mathbb{R}^N\colon ||x|-|y||<\varepsilon\}=B_{|x|+\varepsilon}(0)\setminus 
		\overline{B_{|x|-\varepsilon}}(0)
	\]
	and
	\[
		J_\varepsilon v_{\beta}(x)\coloneqq 2\int\displaylimits_{\mathbb{R}^N\setminus A_\varepsilon(x)}
		\dfrac{\Psi_p(v_{\beta}(x)-v_{\beta}(y))}{|x-y|^{N+sp}} dy.
	\]
	Notice that $J_\varepsilon v_{\beta}(x)$ is finite  for all $x\in\mathbb{R}^N$ since 
	$\beta\in \left(-\tfrac{N}{p-1},\tfrac{ps}{p-1}\right).$

	\medskip
	
	{\it Ideas of the proof of Theorem \ref{theorem:casepsneqN}.}
		We proceed somewhat as in the proof of 
		\cite[Proof of Lemma 3.1]{MR2469027}.
		
		First, we will show that 
		\begin{equation}\label{eq:Jbe}
			J_\varepsilon v_{\beta}(x)=h_{\beta,\varepsilon}(x)|x|^{\beta(p-1)-sp}-
	 		g_{\beta,\varepsilon}(x)\text{ in }\mathbb{R}^N\setminus\{0\},
		\end{equation}
		where
		\begin{align*}
			h_{\beta,\varepsilon}(x)
	 			&\coloneqq\pi\alpha_N\int\displaylimits_{0}^{1-\frac{\varepsilon}{|x|}}
				\Psi_p(1-\rho^{\beta})\left[\rho^{N-1}-\rho^{ps-\beta(p-1)-1}\right]
				G(\rho^{2},N,ps)d\rho,\\
			g_{\beta,\varepsilon}(x)&\coloneqq 4\pi\alpha_N|x|^{{\beta(p-1)-sp}}
	 		\int\displaylimits_{1-\frac{\varepsilon}{|x|}}^{\frac{|x|}{|x|+\varepsilon}}
	 		\Psi_p(1-\rho^{\beta})\rho^{sp-\beta(p-1)-1}G(\rho^{2},N,ps)d\rho.
		\end{align*}
		The function \(G\) is defined in Theorem \ref{Theorem:Fundamental}.
	
		Then we will show that
		\[
			h_{\beta,\varepsilon}(x)|x|^{{\beta(p-1)-sp}}\to \mathcal{C}(\beta)|x|^{{\beta(p-1)-sp}}
			\quad \text{ and } \quad g_{\beta,\varepsilon}(x)\to 0,
		\]
		strongly in $L^1_{\text{loc}}(\mathbb{R}^N\setminus\{0\})$ as $\varepsilon\to 0^+.$ 
		
		Finally, we will conclude our result.
	\medskip
	
	\begin{proof}[Proof of Theorem \ref{theorem:casepsneqN}]
		Throughout this proof, we assume that $ps\neq N,$ 
		$\beta\in\left(-\tfrac{N}{p-1},\tfrac{ps}{p-1}\right),$ $\varepsilon>0$ and 
		$x\in\mathbb{R^N}\setminus\{0\},$ and we simply use the notations
		$G(t)$ instead of $G(t,N,ps).$ We split the proof into four steps.
		
		\medskip
		
		\noindent{\bf Step 1}. 
			The first part of the proof shows \eqref{eq:Jbe}.
			\medskip
		
			We begin by observing that
			\begin{align*}
				J_\varepsilon v_{\beta}(x) &=\displaystyle
				2\int\displaylimits_{\mathbb{R}^N\setminus A_\varepsilon(x)}
				\dfrac{\Psi_p(v_{\beta}(x)-v_{\beta}(y))}{|x-y|^{N+sp}} dy\\
				&=2\int\displaylimits_{\mathbb{R}^N\setminus A_\varepsilon(x)}
				\dfrac{||x|^\beta-|y|^\beta|^{p-2}(|x|^\beta-|y|^\beta)}{|x-y|^{N+sp}} dy\\
				&=2|x|^{\beta(p-1)-sp}
				\int\displaylimits_{\mathbb{R}^N\setminus A_\varepsilon(x)}
				\frac{\left|1-\left(\tfrac{|y|}{|x|}\right)^\beta \right|^{p-2}
				\left(1-\left(\tfrac{|y|}{|x|}\right)^\beta\right)}
				{\left|\tfrac{x}{|x|}-\tfrac{y}{|x|}\right|^{N+sp}} \dfrac{dy}{|x|^N}.\\
			\end{align*}
			Thus, by a simple change of variable and by rotation invariance, we have that
			\[
				J_\varepsilon v_{\beta}(x)=2|x|^{\beta(p-1)-sp}
				\int\displaylimits_{\mathbb{R}^N\setminus A_{\nicefrac{\varepsilon}{|x|}}(e_1)}
				\dfrac{\Psi_p\left(1-|y|^\beta \right)}{\left|e_1-y\right|^{N+sp}} dy,
			\]
			where $e_1=(1,0,\dots,0)\in\mathbb{R}^N.$ 
	
			We now take $y=\rho z$ with $\rho>0$ and 
			$z\in\mathbb{S}^{N-1}\coloneqq\{w\in\mathbb{R}^{N}\colon |w|=1\},$ we get
			\begin{align*}
				&J_\varepsilon v_{\beta}(x)=2|x|^{\beta(p-1)-sp}
				\int\displaylimits_{|1-\rho|\ge\tfrac{\varepsilon}{|x|}}
				\Psi_p\left(1-\rho^\beta \right)
				\int\displaylimits_{\mathbb{S}^{N-1}}\dfrac{d\mathcal{H}^{N-1}(z)}
				{\left|e_1-\rho z\right|^{N+sp}} \rho^{N-1}d\rho\\
				&=2|x|^{\beta(p-1)-sp}
				\int\displaylimits_{|1-\rho|\ge\tfrac{\varepsilon}{|x|}}
				\Psi_p\left(1-\rho^\beta \right)
				\int\displaylimits_{\mathbb{S}^{N-1}}\dfrac{d\mathcal{H}^{N-1}(z)}
				{\left|1-2\rho e_1\cdot z+\rho^2\right|^{\frac{N+sp}2}}
				\rho^{N-1}d\rho.
			\end{align*}
			Using \cite[page 249]{MR2917408} and \cite[3.665 (427)]{MR2360010}, we have that
			\[
				J_\varepsilon v_{\beta}(x)=4\pi\alpha_N|x|^{\beta(p-1)-sp}
				\int\displaylimits_{|1-\rho|\ge\tfrac{\varepsilon}{|x|}}
				\Psi_p\left(1-\rho^\beta \right)\rho^{N-1}\mathcal{K}(\rho) d\rho,
			\]
			where
			\begin{equation}\label{krodef}
	            	\mathcal{K}(\rho)\coloneqq
				\int_0^\pi \dfrac{\sin^{N-2}(\theta)d\theta}
				{\left|1-2\rho \cos(\theta)+\rho^2\right|^{\frac{N+sp}2}}
				=\begin{cases}
						G(\rho^2) &\text{if } \rho<1,\\[5pt]
						\frac{G(\rho^{-2})}{\rho^{N+ps}} &\text{if } \rho>1.
					\end{cases}
            \end{equation}
			Therefore,
			\begin{align*}
				J_\varepsilon & v_{\beta}(x)=4\pi\alpha_N|x|^{\beta(p-1)-sp}
				\left\{
					\int\displaylimits_{1+\frac{\varepsilon}{|x|}}^{\infty}
					\Psi_p\left(1-\rho^\beta\right)\frac{G(\rho^{-2})}{\rho^{ps+1}}d\rho\right.\\
					&\hspace{3.5cm}+\left.\int\displaylimits_{0}^{1-\frac{\varepsilon}{|x|}}
					\Psi_p\left(1-\rho^\beta\right)\rho^{N-1}G(\rho^{2})d\rho
				\right\}\\
				&=4\pi\alpha_N|x|^{\beta(p-1)-sp}
				\left\{
					\int\displaylimits_{0}^{\frac{|x|}{|x|+\varepsilon}}
					\Psi_p\left(1-\rho^{-\beta}\right)\rho^{ps-1}G(\rho^{2})d\rho\right.\\
					&\hspace{3.5cm}+\left.\int\displaylimits_{0}^{1-\frac{\varepsilon}{|x|}}
					\Psi_p\left(1-\rho^\beta\right)\rho^{N-1}G(\rho^{2})d\rho
					\right\}\\
				&=4\pi\alpha_N|x|^{\beta(p-1)-sp}
				\left\{
					\int\displaylimits_{0}^{1-\frac{\varepsilon}{|x|}}
				\Psi_p\left(1-\rho^\beta\right)\left[\rho^{N-1}-\rho^{ps-
					\beta(p-1)-1}\right]G(\rho^{2})d\rho\right.\\
					&\hspace{3.5cm}-
					\left.
					\int\displaylimits_{1-\frac{\varepsilon}{|x|}}^{\frac{|x|}{|x|+\varepsilon}}
					\Psi_p\left(1-\rho^\beta\right)\rho^{ps-\beta(p-1)-1}G(\rho^{2})d\rho
				\right\}.
			\end{align*}
			 So,
			 \[
			 	J_\varepsilon v_{\beta}(x)=h_{\beta,\varepsilon}(x)|x|^{\beta(p-1)-sp}-
			 	g_{\beta,\varepsilon}(x)\text{ in }\mathbb{R}^N\setminus\{0\}.
			 \]

		 \medskip
		 
		\noindent{\bf Step 2}. We now show that
			\[
				h_{\beta,\varepsilon}(x)|x|^{{\beta(p-1)-sp}}\to \mathcal{C}(\beta)|x|^{{\beta(p-1)-sp}}
				\quad \text{ strongly in } 
				L^1_{\text{loc}}(\mathbb{R}^N\setminus\{0\}), 
			\]
			as $\varepsilon\to 0^+.$ 
		
			Given a bounded set $U\subset\mathbb{R}^N\setminus\{0\}$ such that 
				$\overline{U}\subset\mathbb{R}^N\setminus\{0\},$ we want to show that
				\begin{align*}
					&0=\frac1{4\pi\alpha_N}
					    \lim_{\varepsilon\to 0^+}\int\limits_{U} 
					|h_{\beta,\varepsilon}(x)|x|^{{\beta(p-1)-sp}}- \mathcal{C}(\beta)|x|^{{\beta(p-1)-sp}}|dx\\
					&=\frac1{4\pi\alpha_N}
					\lim_{\varepsilon\to 0^+}\int\limits_{U} 
					\left|h_{\beta,\varepsilon}(x)- 
					\mathcal{C}(\beta)|
					\right|x|^{{\beta(p-1)-sp}}dx\\
					&=\lim_{\varepsilon\to 0^+}\int\limits_{U}
					|x|^{{\beta(p-1)-sp}} 
					\left|\int\limits_{1-\frac{\varepsilon}{|x|}}^1 \Psi_p\left(1-\rho^\beta\right)
					(\rho^{N-1}-\rho^{ps-\beta(p-1)-1})G(\rho^2)d\rho\right| dx.
				\end{align*}
		
			Since $\overline{U}\subset\mathbb{R}^N\setminus\{0\}$ is bounded, we have that 
			$|x|^{{\beta(p-1)-sp}} \in L^\infty(U).$ Therefore, it is enough to prove that  
				\begin{equation}\label{eq:step2}
					\lim_{\varepsilon\to 0^+}\int\limits_{U}
					\int\limits_{1-\frac{\varepsilon}{|x|}}^1 |1-\rho^\beta|^{p-1}
						|\rho^{N-1}-\rho^{ps-\beta(p-1)-1}|G(\rho^2)d\rho dx=0.
				\end{equation}
		
			Let $H(\rho)\coloneqq (1-\rho)^{1 +ps} G(\rho^2).$ 
			By \cite[page 271, (2)]{MR0501762}, we have that
			\begin{equation}\label{eq:Hcont}
				\lim_{\rho\to 1^-} H(\rho) 
			\end{equation}
			exists. Then, 
			\[
				|1-\rho^\beta|^{p-1}
				|\rho^{N-1}-\rho^{ps-\beta(p-1)-1}|G(\rho^2)=
				\dfrac{|1-\rho^\beta|^{p-1}
				|\rho^{N-1}-\rho^{ps-\beta(p-1)-1}|}{|1-\rho|^{1+ps}}
				H(\rho),
			\]
			belongs to $L^1(0,1).$ Therefore, by the dominated convergence theorem, we get
			\eqref{eq:step2}.
		
		\medskip
		 
		\noindent{\bf Step 3}. 
			Our next goal is to show that,  
			\begin{equation}\label{eq:step3}
				g_{\beta,\varepsilon}(x)\to 0 \quad\text{strongly in } 
					L^1_{\text{loc}}(\mathbb{R}^N\setminus\{0\}). 
			\end{equation}
			as $\varepsilon\to0^+.$ 
			
			Again, let $U\subset\mathbb{R}^N\setminus\{0\}$ be a bounded set such that 
			$\overline{U}\subset\mathbb{R}^N\setminus\{0\}.$ In this case, we want to show 
			that
			\begin{align*}
				0&=\lim_{\varepsilon\to0^+}\int\limits_U|g_{\beta,\varepsilon}(x)|dx\\
				&=4\pi\alpha_N\lim_{\varepsilon\to0^+}\int\limits_U
				|x|^{{\beta(p-1)-sp}}
		 		\left|\int\displaylimits_{1-\frac{\varepsilon}{|x|}}^{\frac{|x|}{|x|+\varepsilon}}
		 	\Psi_p\left(1-\rho^\beta\right)\rho^{sp-\beta(p-1)-1}G(\rho^{2})d\rho\right| dx.
			\end{align*}
			
			Since $\overline{U}\subset\mathbb{R}^N\setminus\{0\}$ is bounded, we have that 
			$|x|^{{\beta(p-1)-sp}} \in L^\infty(U),$  and 
			\[
				\left|1-\rho^\beta \right|\le C |1-\rho|,
			\]
			where $C$ is a positive constant that depends on $\beta,p,$ and $\text{dist}(U,0).$ Then
			\[
				\int\limits_U|g_{\beta,\varepsilon}(x)|dx
				\le C\int\limits_U
		 		\int\displaylimits_{1-\frac{\varepsilon}{|x|}}^{\frac{|x|}{|x|+\varepsilon}}
		 		\left|1-\rho\right|^{p-1}
				G(\rho^{2})d\rho dx= C\int\limits_U
				\int\displaylimits_{1-\frac{\varepsilon}{|x|}}^{\frac{|x|}{|x|+\varepsilon}}
		 		\dfrac{H(\rho)}{|1-\rho|^{p(s-1)-2}}d\rho dx,
			\]
			where $C$ is a positive constant that depends on $\||x|^{\beta(p-1)-sp}\|_{L\!\prescript{\infty}{}(U)},$ 
			$\beta,p,$ and $\text{dist}(U,0).$
			Therefore, to show \eqref{eq:step3}, it is enough to show that 
			\begin{equation}\label{eq:step3p}
				\lim_{\varepsilon\to0^+}
				\int\limits_U
				\int\displaylimits_{1-\frac{\varepsilon}{|x|}}^{\frac{|x|}{|x|+\varepsilon}}
		 		\dfrac{H(\rho)}{|1-\rho|^{p(s-1)+2}}d\rho dx=0.
			\end{equation}
			
			To prove this, we consider three cases.
			 
			 \medskip
			 
			\noindent{\it Case 1:} $p>\tfrac{1}{1-s}.$ 
			
			By \eqref{eq:Hcont}, $\dfrac{H(\rho)}{|1-\rho|^{p(s-1)+2}}\in L^1(0,1).$
			Thus, by the dominated convergence theorem, we get
			\eqref{eq:step3p}.
				
			\medskip
				
			\noindent{\it Case 2:} We now assume $p<\tfrac{1}{1-s}.$ 
			
			By \cite[pages 257 (5) and 271 (2) ]{MR0501762}, we have that $H$ is differentiable and
			\begin{equation}\label{eq:Hderivada}
				\lim_{\rho\to {1}^-}H^\prime(\rho)
			\end{equation}
			exists.
			Notice that for any $x\in U$ we have
			\begin{align*}
				&\int\displaylimits_{1-\frac{\varepsilon}{|x|}}^{\frac{|x|}{|x|+\varepsilon}}
		 		\dfrac{H(\rho)}{|1-\rho|^{p(s-1)+2}}d\rho 
		 		=\\
		 		&\dfrac1{p(s-1)+1}\left\{
		 		\dfrac{ H\left(\dfrac{|x|}{|x|+\varepsilon}\right)}
		 		{\left(\dfrac{\varepsilon}{|x|+\varepsilon}\right)^{p(s-1)+1}}
		 		-
		 		\dfrac{ H\left(1-\dfrac{\varepsilon}{|x|}\right)}
		 		{\left(\dfrac{\varepsilon}{|x|}\right)^{p(s-1)+1}}
		 		-
		 		\int\displaylimits_{1-\frac{\varepsilon}{|x|}}^{\frac{|x|}{|x|+\varepsilon}}
		 		\dfrac{H^\prime(\rho)}{|1-\rho|^{p(s-1)+1}}d\rho\right\}.
			\end{align*}
			Taking $\delta=\text{dist}(U,0),$ and using \eqref{eq:Hcont} and \eqref{eq:Hderivada}, there
			are two positive constants $C_1$ and $C_2$ such  
			for any $x\in U$ we have
			\begin{align*}
				\left|\dfrac{ H\left(\dfrac{|x|}{|x|+\varepsilon}\right)}
		 		{\left(\dfrac{\varepsilon}{|x|+\varepsilon}\right)^{p(s-1)+1}}
		 		\right.&-
		 		\left.\dfrac{ H\left(1-\dfrac{\varepsilon}{|x|}\right)}
		 		{\left(\dfrac{\varepsilon}{|x|}\right)^{p(s-1)+1}}\right|=\\
		 		&\le\varepsilon^{p(1-s)}
		 		(|x|+\varepsilon)^{p(s-1)+1}
		 			\dfrac{\left|H\left(\dfrac{|x|}{|x|+\varepsilon}\right)
		 			-H\left(1-\dfrac{\varepsilon}{|x|}\right)\right|}{\varepsilon}\\
		 			&\quad+H\left(1-\dfrac{\varepsilon}{|x|}\right)
		 			\dfrac{\left|(|x|+\varepsilon)^{p(s-1)+1}-|x|^{p(s-1)+1}\right|}{\varepsilon}\\
		 		&\le\varepsilon^{p(1-s)}
		 		\Bigg\{
		 			\dfrac{\|H^\prime\|_{{L\!\prescript{\infty}{}(0,1)}}}{\delta^{p(1-s)+1}}
		 			+\dfrac{p(1-s)}{\delta^{p(s-1)}}\|H\|_{{L\!\prescript{\infty}{}(0,1)}}
		 		\Bigg\} \\
		 		&\le C_1\varepsilon^{p(1-s)},
			\end{align*}
			and
			\[
				\int\displaylimits_{1-\frac{\varepsilon}{|x|}}^{\frac{|x|}{|x|+\varepsilon}}
		 		\dfrac{H^\prime(\rho)}{|1-\rho|^{p(s-1)+1}}d\rho\le C_2\varepsilon^{p(1-s)}.
			\]
			
			Then
			\[
				\int\displaylimits_{1-\frac{\varepsilon}{|x|}}^{\frac{|x|}{|x|+\varepsilon}}
		 		\dfrac{H(\rho)}{|1-\rho|^{p(s-1)+2}}d\rho \le C_1\varepsilon^{p(1-s)}+C_2\varepsilon^{p(1-s)} \quad \forall 
		 		x\in U.
			\]
		    This implies \eqref{eq:step3p}.
			
			\medskip
				
			\noindent{\it Case 3:} Finally we consider the case $p=\tfrac{1}{1-s}.$ 
			
			In this case, for any $x\in U$ we have
			\begin{align*}
				&\int\displaylimits_{1-\frac{\varepsilon}{|x|}}^{\frac{|x|}{|x|+\varepsilon}}
		 		\dfrac{H(\rho)}{|1-\rho|^{p(s-1)+2}}d\rho 
		 		=
				\int\displaylimits_{1-\frac{\varepsilon}{|x|}}^{\frac{|x|}{|x|+\varepsilon}}
		 		\dfrac{H(\rho)}{|1-\rho|}d\rho \\
		 		&=	 H\left(\dfrac{|x|}{|x|+\varepsilon}\right)
		 		\log\left(\dfrac{\varepsilon}{|x|+\varepsilon}\right)
		 		-
		 		H\left(1-\dfrac{\varepsilon}{|x|}\right)
		 		{\log\left(\dfrac{\varepsilon}{|x|}\right)}
		 		-
		 		\int\displaylimits_{1-\frac{\varepsilon}{|x|}}^{\frac{|x|}{|x|+\varepsilon}}
		 		H^\prime(\rho)\log(1-\rho)d\rho\\
		 		&=\left\{ H\left(\dfrac{|x|}{|x|+\varepsilon}\right)-H\left(1-\dfrac{\varepsilon}{|x|}\right)
		 		\right\}
		 		\log\left(\dfrac{\varepsilon}{|x|+\varepsilon}\right)\\
		 		&+
		 		H\left(1-\dfrac{\varepsilon}{|x|}\right)
		 		\left\{\log\left(\dfrac{\varepsilon}{|x|+\varepsilon}\right)-
		 		\log\left(\dfrac{\varepsilon}{|x|}\right)\right\}
		 		-
		 		\int\displaylimits_{1-\frac{\varepsilon}{|x|}}^{\frac{|x|}{|x|+\varepsilon}}
		 		H^\prime(\rho)\log(1-\rho)d\rho\\
		 		&=\left\{ H\left(\dfrac{|x|}{|x|+\varepsilon}\right)-H\left(1-\dfrac{\varepsilon}{|x|}\right)
		 		\right\}
		 		\log\left(\dfrac{\varepsilon}{|x|+\varepsilon}\right)
		 		+
		 		H\left(1-\dfrac{\varepsilon}{|x|}\right)\log\left(\dfrac{|x|}{|x|+\varepsilon}\right)\\
		 		&-\int\displaylimits_{1-\frac{\varepsilon}{|x|}}^{\frac{|x|}{|x|+\varepsilon}}
		 		H^\prime(\rho)\log(1-\rho)d\rho.
			\end{align*}
			Again, taking $\delta=\text{dist}(U,0),$ and using \eqref{eq:Hcont} and \eqref{eq:Hderivada}, there
			are three positive constants $C_1,$ $C_2$ and $C_3$ such
			for any $x\in U$ we have
			\begin{align*}
				&\int\displaylimits_{1-\frac{\varepsilon}{|x|}}^{\frac{|x|}{|x|+\varepsilon}}
		 		\dfrac{H(\rho)}{|1-\rho|^{p(s-1)+2}}d\rho 
		 		\le\\
				&\le -\dfrac{\|H^\prime\|_{{L\!\prescript{\infty}{}(0,1)}}}{\delta^2}\varepsilon^2	
				\log\left(\dfrac{\varepsilon}{\delta+\varepsilon}\right)+
		 		\|H \|_{{L\!\prescript{\infty}{}(0,1)}}\dfrac{\varepsilon}{\delta}
		 		-\|H^\prime\|_{{L\!\prescript{\infty}{}(0,1)}}
		 		\int\displaylimits_{1-\frac{\varepsilon}{\delta}}^{1}
		 		\log(1-\rho)d\rho\\
		 		&\le -\dfrac{\|H^\prime\|_{{L\!\prescript{\infty}{}(0,1)}}}{\delta^2}\varepsilon^2	
				\log\left(\dfrac{\varepsilon}{\delta+\varepsilon}\right)+
		 		\|H \|_{{L\!\prescript{\infty}{}(0,1)}}\dfrac{\varepsilon}{\delta}
		 		-\|H^\prime\|_{{L\!\prescript{\infty}{}(0,1)}}
		 		\frac{\varepsilon}{\delta}\left(\log(\varepsilon)-\log(\delta)-1\right)\\
		 		&\le -C_1\varepsilon^2\log\left(\dfrac{\varepsilon}{\delta+\varepsilon}\right)
		 		-C_2\varepsilon\log(\varepsilon)+C_3\varepsilon,
			\end{align*}  
			from which \eqref{eq:step3p} follows.
			
			\medskip
			
		\noindent{\bf Step 4}. Finally, we will show that  
			$v_\beta$ is a weak solution of \eqref{eq:psneqN}.
				
			By step 1, we have that
			\[
				J_\varepsilon v_{\beta}(x)=h_{\beta,\varepsilon}(x)|x|^{\beta(p-1)-sp}-
	 			g_{\beta,\varepsilon}(x)\text{ in }\mathbb{R}^N\setminus\{0\}.
	 		\]
	 		Then, given $\varphi\in C^{\infty}_c(\mathbb{R}^{N}\setminus\{0\})$ we have that
	 		\[
	 			\int\displaylimits_{\mathbb{R}^N} J_\varepsilon v_{\beta}(x) \varphi(x) dx =
	 			\int\displaylimits_{\mathbb{R}^N}\left(h_{\beta,\varepsilon}(x)|x|^{\beta(p-1)-sp}-
	 			g_{\beta,\varepsilon}(x)\right)\varphi(x)\,dx,
	 		\]
	 		that is
	 		\begin{equation}\label{eq:step4_1}
	 			\begin{aligned}
	 				\int\displaylimits_{\mathbb{R}^{2N}}
						(1-\rchi_{A_\varepsilon(x)}(y))&
						\dfrac{\Psi_p(v_{\beta}(x)-v_{\beta}(y))}{|x-y|^{N+sp}}
							\varphi(x) dy dx\\
						&=\frac12\int\displaylimits_{\mathbb{R}^N}
							\left(h_{\beta,\varepsilon}(x)|x|^{\beta(p-1)-sp}-
	 						g_{\beta,\varepsilon}(x)\right)\varphi(x)\,dx.
	 				\end{aligned}
	 			\end{equation}
	 		Interchanging the roles of $x$ and $y$
	 			\begin{equation}\label{eq:step4_2}
	 				\begin{aligned}
	 					\int\displaylimits_{\mathbb{R}^{2N}}
						(1-\rchi_{A_\varepsilon(y)}(x))&
						\dfrac{\Psi_p(v_{\beta}(y)-v_{\beta}(x))}{|x-y|^{N+sp}}
							\varphi(y) dx dy\\
						&=\frac12\int\displaylimits_{\mathbb{R}^N}
							\left(h_{\beta,\varepsilon}(x)|x|^{\beta(p-1)-sp}-
	 						g_{\beta,\varepsilon}(x)\right)\varphi(x)\,dx
	 				\end{aligned}
	 			\end{equation}
	 			
	 		Now, adding \eqref{eq:step4_1} and \eqref{eq:step4_2} and using that
	 			$1-\rchi_{A_\varepsilon(y)}(x)=1-\rchi_{A_\varepsilon(x)}(y)$, we get
	 			\begin{equation}\label{eq:step4_3}
	 				\begin{aligned}
	 					&\int\displaylimits_{\mathbb{R}^N}
								\left(h_{\beta,\varepsilon}(x)|x|^{\beta(p-1)-sp}-
	 						g_{\beta,\varepsilon}(x))\right)\varphi(x)\,dx=\\
	 					&\int\displaylimits_{\mathbb{R}^{2N}}
							(1-\rchi_{A_\varepsilon(x)}(y))
							\dfrac{\Psi_p(v_{\beta}(x)-v_{\beta}(y))(\varphi(x)-\varphi(y))}{|x-y|^{N+sp}}
							 dy dx.
	 				\end{aligned}
	 			\end{equation}
	 			
	 		By steps 2 and 3
	 			\begin{equation}\label{eq:step4_4}
	 				\int\displaylimits_{\mathbb{R}^N}
							\left(h_{\beta,\varepsilon}(x)|x|^{\beta(p-1)-sp}-
	 						g_{\beta,\varepsilon}(x))\right)\varphi(x)\,dx
	 						\to
	 						\mathcal{C}(\beta)
	 						\int\displaylimits_{\mathbb{R}^N}|x|^{\beta(p-1)-sp}\varphi(x) dx.
	 			\end{equation}
                as $\varepsilon\to 0^+.$
                				
			On the other hand, we have that
				\[
					\dfrac{\Psi_p(v_{\beta}(x)-v_{\beta}(y))(\varphi(x)-\varphi(y))}{|x-y|^{N+sp}}
						\in L^{1}(\mathbb{R}^N\times\mathbb{R}^N),
				\]
			and 
				\[
					\rchi_{A_\varepsilon(x)}(y)\to 0
					\text{ a.e. in }\mathbb{R}^{2N} \text{ as }\varepsilon \to 0^+.  
				\]
				
			Then
	 		\begin{equation}\label{eq:step4_5}
	 			\begin{aligned}
	 				\lim_{\varepsilon\to 0^+}
							\int\displaylimits_{\mathbb{R}^{2N}}&\!
						(1-\rchi_{A_\varepsilon(x)}(y))
						\dfrac{\Psi_p(v_{\beta}(x)-v_{\beta}(y))(\varphi(x)-\varphi(y))}{|x-y|^{N+sp}}
							 dy dx\\
	 					&=\int\displaylimits_{\mathbb{R}^{2N}}
						\dfrac{\Psi_p(v_{\beta}(x)-v_{\beta}(y))(\varphi(x)-\varphi(y))}{|x-y|^{N+sp}}
							 dy dx.\\			
	 			\end{aligned}
	 		\end{equation}
	 			
	 		Since $\varphi$ is arbitrary, by \eqref{eq:step4_3}, \eqref{eq:step4_4} and \eqref{eq:step4_5},
	 		we conclude that $v_\beta$ is a weak solution of \eqref{eq:psneqN}.
	\end{proof}

\subsection{Case \texorpdfstring{$ps=N$}{}}
	To complete study of the fundamental solution of the fractional $p-$laplacian, we 
	prove the following result.
	
	\begin{theorem} \label{theorem:caseps=N}
		Let $N\ge2,0<s<1,$ and $1<p<\infty.$ If $ps= N$ then 
		\[
			v(x)=\log(|x|),
		\] 
		is a weak solution of 
		\begin{equation}\label{eq:ps=N}
			(-\Delta_p)^s v(x)= 0\quad\text{in }\mathbb{R}^N\setminus\{0\}.
		\end{equation}
	\end{theorem}
	{\it Ideas of the proof.} 	
	Given $\varepsilon>0,$  we define 
	\[
		J_\varepsilon v(x)\coloneqq 2\int\displaylimits_{\mathbb{R}^N\setminus A_\varepsilon(x)}
		\dfrac{\Psi_p(v(x)-v(y))}{|x-y|^{N+sp}} dy.
	\]
	First, we will show that 
	\[
		J_\varepsilon v(x)\to 0 \text{ strongly in } L^1_{\text{loc}}(\Omega)
	\]
	as $\varepsilon\to 0^+.$ Then, arguing as in  step 4 of the proof of 
	Theorem \ref{theorem:casepsneqN}, we conclude that $v$ is a  weak solution of \eqref{eq:ps=N}.
	
	\begin{proof}[Proof of Theorem \ref{theorem:caseps=N}]
		Throughout this proof, we assume that $ps=N,$ 
		$\varepsilon>0$ and $x\in\mathbb{R^N}\setminus\{0\}.$
		
		Let us begin by observing that
		\begin{align*}
			J_\varepsilon v(x)
				&= 2\int\displaylimits_{\mathbb{R}^N\setminus A_\varepsilon(x)}
				\dfrac{\Psi_p(v(x)-v(y))}{|x-y|^{N+sp}} dy\\
				&=-\dfrac{2}{|x|^N}\int\displaylimits_{\mathbb{R}^N\setminus A_\varepsilon(x)}
				\dfrac{\left|\log\left(\frac{|y|}{|x|}\right)\right|^{p-2}\log\left(\frac{|y|}{|x|}\right)}
				{\left|\frac{x}{|x|}-\frac{y}{|x|}\right|^{2N}} \frac{dy}{|x|^N}.
		\end{align*}
		Now, proceeding as in step 1 of the proof of Theorem \ref{theorem:casepsneqN}
		\[
			J_\varepsilon v(x)
			=-\dfrac{4\alpha_N}{|x|^{N}}
			\int\displaylimits_{|1-\rho|\ge\frac{\varepsilon}{|x|}}
			\Psi_p(\log(\rho))\rho^{N-1}\mathcal{K}(\rho)\, d\rho,
		\]
		where $\mathcal{K}(\rho)$ is defined by \eqref{krodef}.
		
		Then
		\begin{align*}
			J_\varepsilon v(x)=&-\dfrac{4\alpha_N}{|x|^{N}}
			\left\{\int\displaylimits_{0}^{1-\frac{\varepsilon}{|x|}}
			\Psi_p(\log(\rho))\rho^{N-1}G(\rho^2,N,N)\, d\rho\right.\\
			&\quad+\left.\int\displaylimits_{1+\frac{\varepsilon}{|x|}}^{\infty}
			\Psi_p(\log(\rho))\rho^{-N-1}G(\rho^{-2},N,N)\, d\rho\right\}\\
			=&-\dfrac{4\alpha_N}{|x|^{N}}
			\left\{\int\displaylimits_{0}^{1-\frac{\varepsilon}{|x|}}
			\Psi_p(\log(\rho))\rho^{N-1}G(\rho^2,N,N)\, d\rho\right.\\
			&\quad-\left.\int\displaylimits_{0}^{\frac{|x|}{|x|+\varepsilon}}
			\Psi_p(\log(\rho))\rho^{N-1}G(\rho^{2},N,N)\, d\rho\right\}\\
			=&\dfrac{4\alpha_N}{|x|^{N}}
			\int\displaylimits_{1-\frac{\varepsilon}{|x|}}^{\frac{|x|}{|x|+\varepsilon}}
			\Psi_p(\log(\rho))\rho^{N-1}G(\rho^2,N,N)\, d\rho.
		\end{align*}
		That is,
		\begin{equation}\label{eq:Jv}
			J_\varepsilon v(x)=g_\varepsilon(x) \text{ in } \mathbb{R}^N\setminus\{0\}
		\end{equation}
		where
		\[
			g_\varepsilon(x)=\dfrac{4\alpha_N}{|x|^{N}}
			\int\displaylimits_{1-\frac{\varepsilon}{|x|}}^{\frac{|x|}{|x|+\varepsilon}}
			\Psi_p(\log(\rho))\rho^{N-1}G(\rho^2,N,N)\, d\rho.
		\]
		
		We claim 
		\begin{equation}\label{eq:claim}
			g_\varepsilon \to 0 \text{ strongly in } L^1_{\text{loc}}(\mathbb{R}^{N}\setminus\{0\}),
		\end{equation}
		 as $\varepsilon\to 0^+.$
		
		To see this, notice that
		\begin{equation}\label{eq:ge}
			\begin{aligned}
				|g_\varepsilon (x)|\le&
				\dfrac{4\alpha_N}{|x|(|x|+\varepsilon)^{N-1}}
				\int\displaylimits_{1-\frac{\varepsilon}{|x|}}^{\frac{|x|}{|x|+\varepsilon}}
				\dfrac{|\log(\rho)|^{p-1}}{(1-\rho)^{N+1}}(1-\rho)^{N+1}G(\rho^2,N,N)\, d\rho\\
				\le&
				4\alpha_N\dfrac{|x|^{p-2}}{(|x|+\varepsilon)^{N-1}(|x|-\varepsilon)^{p-1}}
				\int\displaylimits_{1-\frac{\varepsilon}{|x|}}^{\frac{|x|}{|x|+\varepsilon}}
				\dfrac{(1-\rho)^{N+1}G(\rho^2,N,N)}{(1-\rho)^{N-p+2}}\, d\rho.
			\end{aligned}
		\end{equation}
		On the other hand, by \cite[pages 257 (5) and 271 (2)]{MR0501762}, we have that 
			$H(\rho)=(1-\rho)^{N+1}G(\rho^2,N,N)$ is differentiable and
			\begin{equation}\label{eq:Hdos}
				\lim_{\rho\to {1}^-}H(\rho) \text{ and } \lim_{\rho\to {1}^-}H^\prime(\rho),
			\end{equation}
			exist.
			
		As in step 3 of the proof of Theorem \ref{theorem:casepsneqN}, we consider three cases
		to prove our claim.
			
		\noindent {\it Case 1:} We start assuming that $p>N+1.$
			
			In this case, owing to \eqref{eq:ge} and \eqref{eq:Hdos}, 
			it is easy to check that there is a positive 
			constant $C$ (independent of $\varepsilon$)
			\begin{align*}
				|g_\varepsilon(x)|\le&
				4\alpha_N\frac{4\alpha_N|x|^{p-2}}{(|x|+\varepsilon)^{N-1}(|x|-\varepsilon)^{p-1}}
				\int\displaylimits_{1-\frac{\varepsilon}{|x|}}^{\frac{|x|}{|x|+\varepsilon}}
				\dfrac{H(\rho)}{(1-\rho)^{N-p+2}}\, d\rho\\
				\le& \frac{4\alpha_N|x|^{p-2}}{(|x|+\varepsilon)^{N-1}(|x|-\varepsilon)^{p-1}}
				\left(\frac{1}{|x|^{p-N-1}}-\frac{1}{(|x|+\varepsilon)^{p-N-1}}\right)
				\varepsilon^{p-N-1}.
			\end{align*}
			This implies \eqref{eq:claim}.
			
			\medskip
		
		\noindent {\it Case 2:} We now study the case $p<N+1.$
			
			 Due to \eqref{eq:Hdos}, we have that
			\begin{align*}
				&\int\displaylimits_{1-\frac{\varepsilon}{|x|}}^{\frac{|x|}{|x|+\varepsilon}}
		 		\dfrac{H(\rho)}{|1-\rho|^{N-p+2}}d\rho 
		 		=\\
		 		&=\dfrac1{N-p+1}\left\{
		 		\dfrac{ H\left(\dfrac{|x|}{|x|+\varepsilon}\right)}
		 		{\left(\dfrac{\varepsilon}{|x|+\varepsilon}\right)^{N-p+1}}
		 		-
		 		\dfrac{ H\left(1-\dfrac{\varepsilon}{|x|}\right)}
		 		{\left(\dfrac{\varepsilon}{|x|}\right)^{N-p+1}}
		 		-
		 		\int\displaylimits_{1-\frac{\varepsilon}{|x|}}^{\frac{|x|}{|x|+\varepsilon}}
		 		\dfrac{H^\prime(\rho)}{|1-\rho|^{N-p+1}}d\rho\right\}\\
		 		&\le
		 		\dfrac{\|H^\prime\|_{{L\!\prescript{\infty}{}(0,1)}}}{N-p+1}
		 		\left\{
		 			\dfrac{\varepsilon}{(|x|+\varepsilon)^{p-N}|x|}
		 			+\dfrac1{p-N}
		 				\left(
		 					\dfrac{1}{|x|^{p-N}}-
		 					\dfrac{1}{(|x|+\varepsilon)^{p-N}}
		 				\right)
		 		\right\}\varepsilon^{p-N}\\
		 		&\qquad+\dfrac{\|H\|_{{L\!\prescript{\infty}{}(0,1)}}}{|x|^{p-N}}\varepsilon^{p-N}.
			\end{align*}
			This implies, using again \eqref{eq:ge},
			\begin{align*}
				&|g_\varepsilon (x)|\le\\
				&C\varepsilon^{p-N}\dfrac{|x|^{p-2}}{(|x|+\varepsilon)^{N-1}(|x|-\varepsilon)^{p-1}}
				\left\{
		 			\dfrac{\varepsilon}{(|x|+\varepsilon)^{p-N}|x|}
		 			+\dfrac{1}{|x|^{p-N}}-
		 					\dfrac{1}{(|x|+\varepsilon)^{p-N}}
		 		\right\}
			\end{align*}
			where $C$ is a positive constant independent of $\varepsilon.$
			Now it is easy to check \eqref{eq:claim}.
			
			\medskip
	
		\noindent {\it Case 3:} To conclude the proof of our claim, we study the case $p=N+1.$
		
			Again by \eqref{eq:Hdos}, we have that	
			\begin{align*}
				&\int\displaylimits_{1-\frac{\varepsilon}{|x|}}^{\frac{|x|}{|x|+\varepsilon}}
		 		\dfrac{H(\rho)}{|1-\rho|^{N-p+2}}d\rho 
		 		=\int\displaylimits_{1-\frac{\varepsilon}{|x|}}^{\frac{|x|}{|x|+\varepsilon}}
		 		\dfrac{H(\rho)}{|1-\rho|}d\rho= \\
		 		&
		 		 H\left(\dfrac{|x|}{|x|+\varepsilon}\right)
		 		\log\left(\dfrac{\varepsilon}{|x|+\varepsilon}\right)
		 		-
		 		{ H\left(1-\dfrac{\varepsilon}{|x|}\right)}
		 		{\log\left(\dfrac{\varepsilon}{|x|}\right)}
		 		-
		 		\int\displaylimits_{1-\frac{\varepsilon}{|x|}}^{\frac{|x|}{|x|+\varepsilon}}
		 		{H^\prime(\rho)}{\log(1-\rho)}d\rho\\
		 		&\le
		 		{\|H^\prime\|_{{L\!\prescript{\infty}{}(0,1)}}}
		 		\left\{
		 			-\dfrac{\varepsilon\log\left(\frac{\varepsilon}{|x|+\varepsilon}\right)}
		 			{(|x|+\varepsilon)|x|}		
		 			-
		 				\left(
		 					\dfrac{\varepsilon\left(\log(\varepsilon)-1\right)}{|x|(|x|+\varepsilon)}
		 					+\dfrac{\log(|x|+\varepsilon)}{(|x|+\varepsilon)}
		 					-\dfrac{\log(|x|)}{|x|}
		 				\right)
		 		\right\}\varepsilon\\
		 		&\qquad+\dfrac{\|H\|_{{L\!\prescript{\infty}{}(0,1)}}}{|x|}\varepsilon.
			\end{align*}
			This implies, using again \eqref{eq:ge},
			\begin{align*}
				&|g_\varepsilon (x)|\le\\
				&C\varepsilon |x|^{N-2}
				\frac{
		 			-2\varepsilon\log\left(\varepsilon\right)
		 			+2\varepsilon+
		 					(\varepsilon-|x|)\log(|x|+\varepsilon)
		 					+(|x|+\varepsilon)\log(|x|)+|x|}
		 		{(|x|+\varepsilon)^{N}(|x|-\varepsilon)^{N}}
			\end{align*}
			where $C$ is a positive constant independent of $\varepsilon.$
			This implies \eqref{eq:claim}.
			
		\medskip
		
		Finally, we prove that $v(x)=\log(|x|),$ is a weak solution of \eqref{eq:ps=N}.
		By \eqref{eq:Jv}, we have that
			\[
				J_\varepsilon v(x)=g_{\varepsilon}(x)\text{ in }\mathbb{R}^N\setminus\{0\}.
	 		\]
	 		Then, given $\varphi\in C^{\infty}_c(\mathbb{R}^{N}\setminus\{0\})$ we have that
	 		\[
	 			\int\displaylimits_{\mathbb{R}^N} J_\varepsilon v(x) \varphi(x) dx =
	 			\int\displaylimits_{\mathbb{R}^N}g_{\varepsilon}(x)\varphi(x)\,dx
	 		\]
	 		that is
	 		\begin{equation}\label{eq:ps=N_1}
	 				\int\displaylimits_{\mathbb{R}^{2N}}
						(1-\rchi_{A_\varepsilon(x)}(y))
						\dfrac{\Psi_p(v(x)-v(y))}{|x-y|^{N+sp}}
							\varphi(x) dy dx
							=\frac12
                            \int\displaylimits_{\mathbb{R}^N}
	 						g_{\varepsilon}(x)\varphi(x)dx.
	 			\end{equation}
	 		Interchanging the roles of $x$ and $y$
	 			\begin{equation}\label{eq:ps=N_2}
	 					\int\displaylimits_{\mathbb{R}^{2N}}
						(1-\rchi_{A_\varepsilon(y)}(x))
						\dfrac{\Psi_p(v(y)-v(x))}{|x-y|^{N+sp}}
							\varphi(y) dx dy=\frac12\int\displaylimits_{\mathbb{R}^N}
	 						g_{\varepsilon}(x)\varphi(x)\,dx
	 			\end{equation}
	 			
	 		Now, adding \eqref{eq:ps=N_1} and \eqref{eq:ps=N_2} and using that
	 			$1-\rchi_{A_\varepsilon(y)}(x)=1-\rchi_{A_\varepsilon(x)}(y)$, we get
	 			\begin{equation}\label{eq:ps=N_3}
						\int\displaylimits_{\mathbb{R}^N}
								g_{\varepsilon}(x)\varphi(x)\,dx=
	 						\int\displaylimits_{\mathbb{R}^{2N}}
							(1-\rchi_{A_\varepsilon(x)}(y))
							\dfrac{\Psi_p(v(x)-v(y))(\varphi(x)-\varphi(y))}{|x-y|^{N+sp}}
							 dy dx.
	 			\end{equation}
	 			
	 		By \eqref{eq:claim} we have that
	 			\begin{equation}\label{eq:ps=N_4}
	 				\lim_{\varepsilon\to 0^+}\int\displaylimits_{\mathbb{R}^N}
	 						g_{\varepsilon}(x)\varphi(x)\,dx
	 						=0.
	 			\end{equation}
				
			On the other hand, we have that
				\[
					\dfrac{\Psi_p(v(x)-v(y))(\varphi(x)-\varphi(y))}{|x-y|^{N+sp}}
						\in L^{1}(\mathbb{R}^{2N})
				\]
			and 
				\[
					\rchi_{A_\varepsilon(x)}(y)\to 0
					\text{ a.e. in }\mathbb{R}^{2N} 
					\text{ as }\varepsilon \to 0^+.  
				\]

			Then
	 		\begin{equation}\label{eq:ps=N_5}
	 			\begin{aligned}
	 				\lim_{\varepsilon\to 0^+}
							\int\displaylimits_{\mathbb{R}^{2N}}&
						(1-\rchi_{A_\varepsilon(x)}(y))
						\dfrac{\Psi_p(v(x)-v(y))(\varphi(x)-\varphi(y))}{|x-y|^{N+sp}}
							 dy dx\\
	 					&=\int\displaylimits_{\mathbb{R}^{2N}}
						\dfrac{\Psi_p(v(x)-v(y))(\varphi(x)-\varphi(y))}{|x-y|^{N+sp}}
							 dy dx.\\			
	 			\end{aligned}
	 		\end{equation}
	 			
	 		Since $\varphi$ is arbitrary, by \eqref{eq:ps=N_3}, \eqref{eq:ps=N_4} and \eqref{eq:ps=N_5},
	 		we conclude that $v_\beta$ is a weak solution of \eqref{eq:psneqN}.
	\end{proof}

\section{Hadamard properties}\label{hp}
	To prove our Hadamard properties, we use comparison techniques that require modifying the 
	fundamental solution near the origin to put it below a 
	weak fractional superharmonic function near the origin. 	

\subsection{Two fractional subharmonic functions}
	We start by building  two weak
	fractional subharmonic functions from the fundamental solution.
	
	\medskip
	
	Throughout this section $N\ge2,0<s<1,$  $1<p<\infty,$ $N>ps,$ 
	\mbox{$\beta\in\left(-\tfrac{N}{p-1},\tfrac{ps-N}{p-1}\right)$} and
	$0<\varepsilon<1<r<R.$ Now, define
	\[
		B_\rho\coloneqq B_\rho(0) \quad(\rho>0), \quad 
		A_{r,R}\coloneqq\left\{x\in\mathbb{R}^N\colon r<|x|<R\right\},
	\]
	and
	\[	
		\phi_\varepsilon(x)\coloneqq
		\begin{cases}
			\varepsilon^\beta&\text{if } 0\le |x|<\varepsilon,\\
			|x|^\beta &\text{if } \varepsilon\le |x|.
		\end{cases}	
	\]
	
	\begin{lemma}\label{lemma:aux1}
		There exists $\varepsilon_0\in(0,1)$ independent of $R$ such that for any $\varepsilon\in(0,\varepsilon_0),$
		$\phi_\varepsilon$ is a weak solution of 
		\[
			(-\Delta_p)^s \phi_\varepsilon(x)\le 0\text{ in } A_{r,R}.
		\]
	\end{lemma}
	\begin{proof}
		Let $\varphi\in C_c^\infty(A_{r,R})$ be non-negative.Then
		\begin{equation}\label{eq:l1_1}		
			\int\displaylimits_{\mathbb{R}^{2N}}
			\frac{\Psi_p(\phi_{\varepsilon}(x)-\phi_{\varepsilon}(y))(\varphi(x)-\varphi(y))}
			{|x-y|^{N+sp}} dxdy=I_1+I_2,
		\end{equation}
		where
		\begin{align*}
			I_1&=2\int\displaylimits_{B_\varepsilon}\int\displaylimits_{\mathbb{R}^N\setminus B_\varepsilon}
			\frac{\Psi_p(\phi_{\varepsilon}(x)-\phi_{\varepsilon}(y))(\varphi(x)-\varphi(y))}{|x-y|^{N+sp}} dxdy\\
			&=-2\int\displaylimits_{\mathbb{R}^N\setminus B_r}
			\int\displaylimits_{B_\varepsilon} \frac{|\varepsilon^\beta-|x|^\beta|^{p-1}}{|x-y|^{N+sp}} dy 
			\,\varphi(x) dx
		\end{align*}
		and
		\begin{align*}
			I_2&=\int\displaylimits_{(\mathbb{R}^N\setminus B_\varepsilon)^2}
			\frac{\Psi_p(\phi_{\varepsilon}(x)-\phi_{\varepsilon}(y))(\varphi(x)-\varphi(y))}{|x-y|^{N+sp}} dxdy\\
			&=
			\int
			\displaylimits_{(\mathbb{R}^N\setminus B_\varepsilon)^2}
			\frac{\Psi_p(|x|^\beta-|y|^\beta)
			(\varphi(x)-\varphi(y))}{|x-y|^{N+sp}} dxdy\\
			&=\int\displaylimits_{\mathbb{R}^{2N}}			
			\frac{\Psi_p(|x|^\beta-|y|^\beta)(\varphi(x)-\varphi(y))}{|x-y|^{N+sp}} dxdy\\
			&\qquad+ 2\int\displaylimits_{\mathbb{R}^N\setminus B_r}
			\int\displaylimits_{B_\varepsilon} \frac{||y|^\beta-|x|^\beta|^{p-1}}{|x-y|^{N+sp}} dy\, \varphi(x) dx.
		\end{align*}
		
		By Theorem \ref{theorem:casepsneqN}, we have that
		\[
			I_2=\mathcal{C}(\beta)\int\displaylimits_{\mathbb{R}^N}|x|^{\beta(p-1)-sp}\varphi(x) dx +2\int\displaylimits_{\mathbb{R}^N\setminus B_r}
			\int\displaylimits_{B_\varepsilon} \frac{||y|^\beta-|x|^\beta|^{p-1}}{|x-y|^{N+sp}} dy\, \varphi(x) dx.
		\]
		Then
		\begin{equation}\label{eq:l1_2}
				I_1+I_2=\mathcal{C}(\beta)\int\displaylimits_{\mathbb{R}^N}|x|^{\beta(p-1)-sp}\varphi(x) dx +2 
				\int\displaylimits_{\mathbb{R}^N\setminus B_r} F(x)\varphi(x) dx,
		\end{equation}
		where
		\[
			 F(x)\coloneqq
			\int\displaylimits_{B_\varepsilon} \frac{||y|^\beta-|x|^\beta|^{p-1}-
			|\varepsilon^\beta-|x|^\beta|^{p-1}}{|x-y|^{N+sp}} dy=\int\displaylimits_{B_\varepsilon} 
			k(x,\tfrac{y}{|x|}) dy\,|x|^{\beta(p-1)-N-sp}
		\]
		with
		\[
			k(x,z)\coloneqq \frac{\left|z^{\beta}-1\right|^{p-1}-
			\left|\left(\tfrac{\varepsilon}{|x|}\right)^\beta-1\right|^{p-1}}{\left|\tfrac{x}{|x|}-z
			\right|^{N+sp}}. 
		\]
		Making the change of variables $z=\frac{y}{|x|},$ we have that
		\begin{equation}\label{eq:l1_3}
	     F(x)=\int\displaylimits_{B_\varepsilon} k\left(x,\tfrac{y}{|x|}\right) dy\,|x|^{\beta(p-1)-N-sp}
			=\int\displaylimits_{B_{\frac{\varepsilon}{|x|}}} k(x,z) dz\,|x|^{\beta(p-1)-sp}.
		\end{equation}
		
		On the other hand, since $|x|>r$ and $\beta<0,$ we have that 
		$\left(\tfrac{\varepsilon}{|x|}\right)^\beta\ge \left(\tfrac{\varepsilon}{r}\right)^\beta,$ and
		therefore 
		\[
			k(x,z)\le \frac{\left||z|^{\beta}-1\right|^{p-1}-
			\left|\left(\tfrac{\varepsilon}{r}\right)^\beta-1\right|^{p-1}}{\left|\tfrac{x}{|x|}-z
			\right|^{N+sp}}. 
		\]
		Thus, by a simple change of variable and by rotation invariance, we have that
		\begin{equation}\label{eq:l1_4}
			0\le \int\displaylimits_{B_{\frac{\varepsilon}{|x|}}} k(x,z) dz\le
			\int\displaylimits_{B_{\frac{\varepsilon}{|x|}}} \frac{\left||z|^{\beta}-1\right|^{p-1}-
			\left|\left(\tfrac{\varepsilon}{r}\right)^\beta-1\right|^{p-1}}{\left|e_1-z
			\right|^{N+sp}} dz.
		\end{equation}
		
		Now, since $|e_1-z|\ge 1-\tfrac{\varepsilon}{r}$ for any $z\in B_{\frac{\varepsilon}{|x|}},$
		$p>1,$ $\beta\in\left(-\tfrac{N}{p-1},\tfrac{ps-N}{p-1}\right),$ we
		get
		\begin{equation}\label{eq:l1_5}
			\begin{aligned}
				0&\le
				\int\displaylimits_{B_{\frac{\varepsilon}{|x|}}} \frac{\left||z|^{\beta}-1\right|^{p-1}-
				\left|\left(\tfrac{\varepsilon}{r}\right)^\beta-1\right|^{p-1}}{\left|e_1-z
				\right|^{N+sp}} dz\\
				&\le\frac{1}{\left(1-\tfrac{\varepsilon}{r}\right)^{N+sp}}
				\int\displaylimits_{B_{\frac{\varepsilon}{|x|}}}|z|^{\beta(p-1)} dz=
				\frac{\alpha_N}{\left(1-\tfrac{\varepsilon}{r}\right)^{N+sp}}
				\left(\tfrac{\varepsilon}{r}\right)^{\beta(p-1)+N}.
			\end{aligned}
		\end{equation}
		Then, by \eqref{eq:l1_3}, \eqref{eq:l1_4},
		and \eqref{eq:l1_5}, we have that
		\begin{equation}\label{eq:l1_6}	
			0\le F(x)\le \frac{\alpha_N}{\left(1-\tfrac{\varepsilon}{r}\right)^{N+sp}}
				\left(\tfrac{\varepsilon}{r}\right)^{\beta(p-1)+N}|x|^{\beta(p-1)-sp},
		\end{equation}
		for any $x\in\mathbb{R}^N\setminus B_r.$
		
		By \eqref{eq:l1_1}, \eqref{eq:l1_2}, and \eqref{eq:l1_6},  
		we obtain
		\begin{align*}
			&\int\displaylimits_{\mathbb{R}^{2N}}
			\frac{\Psi_p
			(\phi_{\varepsilon}(x)-\phi_{\varepsilon}(y))(\varphi(x)-\varphi(y))}{|x-y|^{N+sp}} dxdy\le\\
			&\qquad\le\left(\mathcal{C}(\beta)+
					\frac{2\alpha_N}{\left(1-\tfrac{\varepsilon}{r}\right)^{N+sp}}
				\left(\tfrac{\varepsilon}{r}\right)^{\beta(p-1)+N}\right)
				\int\displaylimits_{\mathbb{R}^N\setminus B_r} |x|^{\beta(p-1)-sp}\varphi(x) dx.
		\end{align*}
		Thus, by \eqref{signodecbeta}, we get
		\[
			\int\displaylimits_{\mathbb{R}^{2N}}
			\frac{\Psi_p
			(\phi_{\varepsilon}(x)-\phi_{\varepsilon}(y))(\varphi(x)-\varphi(y))}{|x-y|^{N+sp}} dxdy
			\le 0
		\]
		for any $\varepsilon$ small enough.
	\end{proof}

	\begin{remark}
		It is worth noting that although Lemma \ref{lemma:aux1} establishes that 
		$\varphi_\varepsilon(x) = \max\{|x|^\beta, \varepsilon^\beta\}$ is $(-\Delta_p)^s$ 
		non-positive in the \emph{weak} sense, a similar conclusion holds in the viscosity sense by standard arguments. Specifically, since both $|x|^\beta$ and 
		$\varepsilon^\beta$ are $(-\Delta_p)^s$ non-positive in the viscosity sense, 
		$\varphi_\epsilon$ inherits this property. Any viscosity test function of 
		$\varphi_\epsilon$ can be viewed as a test function for either $|x|^\beta$ or $\epsilon^\beta$, ensuring the desired inequality. This connection helps 
		bridge the weak and viscosity frameworks and aligns with intuition from viscosity theory.
	\end{remark}

	Before showing our next result, we define
	\[
		\mathcal{A}_r\coloneqq A_{\tfrac{r}2,2r}, \quad
		r_\varepsilon\coloneqq r\left(\frac{\varepsilon}{1+\varepsilon 2^\beta}\right)^{-\tfrac1\beta},
	\]
	and
	\[
		\psi_\varepsilon(x)\coloneqq\begin{cases}
						r_\varepsilon^\beta&\text{if } 0\le |x|\le r_{\varepsilon},\\
						|x|^\beta &\text{if }  r_{\varepsilon}<|x|\le 2r,\\
						(2r)^\beta &\text{if }  2r<|x|.
					\end{cases}
	\]
	Observe that $r_\varepsilon<\tfrac{r}2$ and $r_\varepsilon\to 0^+$ as $\varepsilon\to 0^+.$
	
	\begin{lemma}\label{lemma:aux2}
		There is $\varepsilon_0\in(0,1)$ independent of $r$ 
		such that for any $\varepsilon\in(0,\varepsilon_0),$
		$\psi_\varepsilon$ is a weak solution of 
		\[
			(-\Delta_p)^s \psi_\varepsilon(x)\le 0\text{ in } \mathcal{A}_{r}.
		\]
	\end{lemma}
	
	\begin{proof}
		Let $\varphi\in C_c^\infty(\mathcal{A}_r)$ be non-negative. Then
		\begin{equation}\label{eq:l2_1}		
			\int\displaylimits_{\mathbb{R}^{2N}}
			\frac{\Psi_p
			(\psi_{\varepsilon}(x)-\psi_{\varepsilon}(y))(\varphi(x)-\varphi(y))}{|x-y|^{N+sp}} dxdy=
			J_1+2J_2,
		\end{equation}
		where
		\begin{align*}		
			J_1&=\int\displaylimits_{\mathcal{A}_r^2}\frac{\Psi_p
					(\psi_{\varepsilon}(x)-\psi_{\varepsilon}(y))(\varphi(x)-\varphi(y))}{|x-y|^{N+sp}} dxdy,\\
			J_2&=\int\displaylimits_{\mathcal{A}_r}\int\displaylimits_{\mathbb{R}^N\setminus \mathcal{A}_r}
					\frac{\Psi_p
					(\psi_{\varepsilon}(x)-\psi_{\varepsilon}(y))}{|x-y|^{N+sp}} \varphi(x) dydx.
		\end{align*}

		Observe that
		\[
			\mathbb{R^N}\setminus \mathcal{A}_r=C_r^1\cup C_r^2\cup C_r^3
		\]
		with
		\[
			C_r^1\coloneqq A_{r_\varepsilon,\tfrac{r}{2}},\quad C_r^2\coloneqq\mathbb{R}^N\setminus B_{2r}
			\quad\text{and}\quad C_r^3\coloneqq B_{r_\varepsilon}.
		\]
		Then
		\begin{align*}
			J_2&=\int\displaylimits_{\mathcal{A}_r}\int\displaylimits_{C_r^1}
			\dfrac{\Psi_p(|x|^\beta-|y|^\beta)}{|x-y|^{N+sp}}dy
			\,\varphi(x) dx\\
			&+\int\displaylimits_{\mathcal{A}_r}\int\displaylimits_{C_r^2}\dfrac{||x|^\beta-(2r)^\beta|^{p-1}}{|x-y|^{N+sp}}dy
			\,\varphi(x) dx
			-\int\displaylimits_{\mathcal{A}_r}\int\displaylimits_{C_r^3}\dfrac{|r_\varepsilon^\beta-|x|^\beta|^{p-1}}{|x-y|^{N+sp}}dy
			\,\varphi(x) dx.
		\end{align*}
		Thus
		\[
				J_1+2J_2=
				\int\displaylimits_{\mathbb{R}^{2N}}\frac{
					\Psi_p(|x|^\beta-|y|^\beta)(\varphi(x)-\varphi(y))}{|x-y|^{N+sp}} dxdy+2\int\displaylimits_{\mathcal{A}_r}\left(G_1(x)+G_2(x)\right)\varphi(x)dx,
		\]
		where
		\begin{align*}
			G_1(x)&=\int\displaylimits_{C^2_r}\dfrac{||x|^\beta-(2r)^\beta|^{p-1}-||x|^\beta-|y|^\beta|^{p-1}}{|x-y|^{N+sp}}dy,\\
			G_2(x)&=\int\displaylimits_{C^3_r}\dfrac{||y|^\beta-|x|^\beta|^{p-1}
			-|r_\varepsilon^\beta-|x|^\beta|^{p-1}}{|x-y|^{N+sp}}dy.
		\end{align*}
		
		By Theorem \ref{theorem:casepsneqN}, we have
		\begin{equation}\label{eq:l2_2}
			J_1+J_2\le 2\int\displaylimits_{\mathcal{A}_r}\left(G_1(x)+G_2(x)\right)\varphi(x)dx.
		\end{equation}
		
		On the other hand, for any $x\in \mathcal{A}_r$ 
		\begin{equation}\label{eq:l2_3}
			G_2(x)\le \alpha_N \dfrac{r_\varepsilon^{\beta(p-1)+N}}{r^{N+sp}},
		\end{equation}
		by calculation as in the proof of Lemma \ref{lemma:aux1}. 
		
		Now, taking $x\in \mathcal{A}_r,$ we get
		\begin{align*}
			G_1(x)&=\int\displaylimits_{C^2_r}\dfrac{||x|^\beta-(2r)^\beta|^{p-1}-||x|^\beta-|y|^\beta|^{p-1}}{|x-y|^{N+sp}}dy\\
			&=\int\displaylimits_{C^2_r}\dfrac{\left|1-\left(\tfrac{2r}{|x|}\right)^\beta\right|^{p-1}
			-\left|1-\left(\tfrac{|y|}{|x|}\right)^\beta\right|^{p-1}}{\left|\tfrac{x}{|x|}-\tfrac{y}{|x|}
			\right|^{N+sp}}dy|x|^{\beta(p-1)-N-sp}.
		\end{align*}	
		Making the change of variable $z=\tfrac{y}{|x|},$ we obtain that
		\[
			G_1(x)=\int\displaylimits_{\mathbb{R}^N\setminus B_{\frac{2r}{|x|}}}
			\dfrac{\left|1-\left(\tfrac{2r}{|x|}\right)^\beta\right|^{p-1}
			-\left|1-|z|^\beta\right|^{p-1}}{\left|\tfrac{x}{|x|}-z
			\right|^{N+sp}}dy|x|^{\beta(p-1)-sp}.
		\]
		Since $x\in \mathcal{A}_r$ and $\beta<0,$ we have that $2<\tfrac{2r}{|x|}<4$ and
		\begin{align*}
			G_1(x)
			&\le\int\displaylimits_{\mathbb{R}^N
				\setminus B_{4}}
			\dfrac{\left|1-\left(\tfrac{2r}{|x|}\right)^\beta\right|^{p-1}
			-\left|1-|z|^\beta\right|^{p-1}}{\left|\tfrac{x}{|x|}-z
			\right|^{N+sp}}dy\left(\frac{r}{2}\right)^{\beta(p-1)-sp}\\
			&\le 
			\int\displaylimits_{\mathbb{R}^N\setminus B_{4}}
			\dfrac{\left|1-4^\beta\right|^{p-1}
				-\left|1-|z|^\beta\right|^{p-1}}{\left|\tfrac{x}{|x|}-z
				\right|^{N+sp}}dy\left(\frac{r}{2}\right)^{\beta(p-1)-sp}
		\end{align*}
		As the last integration is invariant under the rotation of coordinate axes, we get
		\begin{equation}\label{eq:l2_4}
			G_1(x)\le D_N r^{\beta(p-1)-sp} \quad \forall x\in \mathcal{A}_r,
		\end{equation}
		here $D_N$ denotes a negative constant depending only on N.
		
		Finally by \eqref{eq:l2_1}, \eqref{eq:l2_2}, \eqref{eq:l2_4} and \eqref{eq:l2_3},
		\begin{align*}
			\int\displaylimits_{\mathbb{R}^{2N}}&\frac{\Psi_p
			(\psi_{\varepsilon}(x)-\psi_{\varepsilon}(y))(\varphi(x)-\varphi(y))}{|x-y|^{N+sp}} dxdy\\
			&\le \left(D_N r^{\beta(p-1)-sp} + \alpha_N \dfrac{r_\varepsilon^{\beta(p-1)+N}}{r^{N+sp}}\right)
			\int\displaylimits_{\mathcal{A}_r}\varphi(x)\, dx\\
			&\le \left(D_N  + \alpha_N 
			\left(\frac{\varepsilon}{1+\varepsilon 2^\beta}\right)^{-(p-1)-\tfrac{N}\beta}\right)r^{\beta(p-1)-sp}
			\int\displaylimits_{\mathcal{A}_r}\varphi(x)\, dx.
		\end{align*}
		Since $D_N<0$ and $\left(\tfrac{\varepsilon}{1+\varepsilon 2^\beta}\right)^{-(p-1)-\tfrac{N}\beta}\to0^+$ 
		as $\varepsilon\to0^,$ there is a positive $\varepsilon_0$ independent of $r$ such that
		\[
			\int\displaylimits_{\mathbb{R}^{2N}}\frac{\Psi_p
			(\psi_{\varepsilon}(x)-\psi_{\varepsilon}(y))(\varphi(x)-\varphi(y))}{|x-y|^{N+sp}} dxdy\le0.
		\]
	\end{proof}

\subsection{Hadamard-type results}
	We are now in a position to prove our Hadamard properties.
	\begin{lemma}\label{lemma:Had1}
		Let $N\ge2,$ $0<s<1$ and $1<p<\infty.$ If $N>ps$ then for 
		any $\beta\in\left(-\tfrac{N}{p-1},\tfrac{ps-N}{p-1}\right)$
		and any $r_0>1$ there is a positive constant  $C>0$ such that for any $u\not\equiv0$  
		non-negative lower semi-continuous weak solution  of
		\[
			(-\Delta_p)^su\ge0\quad\text{in }\mathbb{R}^N,
		\]
		we have that
		\[
			m(r)\ge Cm(r_0)r^\beta \quad\forall r>r_0,
		\]
		where $m(r)\coloneqq \min\{u(x)\colon x\in \overline{B_r}(0)\}.$
	\end{lemma}
	\begin{proof}
		Let $r_0>1.$ By Lemma \ref{lemma:aux1}, there exist $\varepsilon_0>0$ such that
		$\phi_\varepsilon$ is a weak solution of
		\begin{equation}
			\label{eq:Had1_1}
				(-\Delta_p)^s \phi_\varepsilon(x)\le0 \quad\text{in } A_{r_0,R}
		\end{equation}
		for any $R>r_0.$ For any $\varepsilon>\varepsilon_0$ and $R>r_0,$ we define
		\[
			H_{\varepsilon,R}(x)\coloneqq \dfrac{m(r_0)}{\varepsilon^\beta-R^\beta}
			\begin{cases}
				\phi_\varepsilon(x)-R^\beta&\text{if } |x|\ge \varepsilon,\\
				0&\text{if } |x|\le \varepsilon.
			\end{cases}
		\]
		 Observe that $H_{\varepsilon,R}$ is a weak solution of \eqref{eq:Had1_1}
		 and $u\ge H_{\varepsilon,R}$ in $\mathbb{R}^N\setminus A_{r_0,R}.$ 
		 Thus, by the comparison principle (see, for instance, \cite{MR3631323}), we have that
		 $u\ge H_{\varepsilon,R}$ in $\mathbb{R}^N.$ Then
		 \[
		 	u(x)\ge \dfrac{m(r_0)}{\varepsilon^\beta-R^\beta} (|x|^\beta-R^\beta) \text{ in }
		 	A_{r_0,R}.
		 \] 
		 Passing to limits as $R\to\infty,$ we have that
		 \[
		 	u(x)\ge \frac{m(r_0)}{\varepsilon^\beta}|x|^\beta \text{ in } \{x\in\mathbb{R}^N\colon |x|>r_0\}.
		 \]
		 Finally, taking $C=\varepsilon_0^{-\beta}$ we have
		 \[
		 	m(r)\ge Cm(r_0)r^\beta\quad\forall r\ge r_0.
		 \]
	\end{proof}
	
	Our second Hadamard property is
	
	\begin{lemma}\label{lemma:Had2}
		Let $N\ge2,$ $0<s<1$ and $1<p<\infty.$ If $N>ps$ then 
		there is a positive constant  $C>0$ such that for any  $u\not\equiv0$
		non-negative lower semi-continuous weak solution  of
		\[
			(-\Delta_p)^su\ge 0\quad\text{in }\mathbb{R}^N
		\]
		we have that
		\[
			m(\tfrac{r}2)\le Cm(r) \quad\forall r>1,
		\]
		where $m(r)\coloneqq \min\{u(x)\colon x\in \overline{B_r}(0)\}.$
	\end{lemma}
	\begin{proof}
		Let $r>1$ and 
		$\beta\in\left(-\tfrac{N}{p-1},\tfrac{ps-N}{p-1}\right).$ 
		By Lemma \ref{lemma:aux2}, there is $\varepsilon_0>0$ independent of $r$ 
		such that $\psi_\varepsilon$ is a weak solution of
		\begin{equation}\label{eq:Had2_1}
				(-\Delta_p)^s \psi_\varepsilon(x)\le0 \quad\text{in } \mathcal{A}_r.
		\end{equation}
		For any $\varepsilon>\varepsilon_0,$  we define
		\[
			J_{\varepsilon,r}(x)\coloneqq m\left(\tfrac{r}2\right)
			\dfrac{\psi_\varepsilon(x)-(2r)^\beta}{r_\varepsilon^\beta-(2r)^\beta}.
		\]
		 Observe that $J_{\varepsilon,r}$ is also weak solution of \eqref{eq:Had2_1}
		 and $u\ge J_{\varepsilon,r}$ in $\mathbb{R}^N\setminus \mathcal{A}_r.$ 
		 Thus, by the comparison principle, we have that
		 $u\ge J_{\varepsilon,r}$ in $\mathbb{R}^N.$ Then
		 \begin{align*}
			m(r)&\ge m\left(\tfrac{r}2\right)
			\min\left\{\dfrac{\psi_\varepsilon(x)-(2r)^\beta}{r_\varepsilon^\beta-(2r)^\beta}\colon
			|x|=r\right\}=m\left(\tfrac{r}2\right)\frac{r^\beta-2^\beta r^\beta}{r^\beta
			\left(\frac{\varepsilon}{1+\varepsilon 2^\beta}\right)^{-1}-2^\beta r^\beta}\\
			&=	m\left(\tfrac{r}2\right)\varepsilon(1-2^\beta).
		\end{align*}
		 Finally, taking $C=\varepsilon (1-2^\beta)$ we have
		 \[
		 	m(r)\ge Cm\left(\tfrac{r}2\right).
		 \]
	\end{proof}
\section{A Liouville-type theorem}\label{st1}
	We now prove a Liouville-type theorem. We split the proof into two cases.
	
\subsection{Case \texorpdfstring{$N<sp$}{N<sp}} Let $\beta\in\left(0,\tfrac{ps-N}{p-1}\right),$ 
		$0<\varepsilon<1<r<R,$ and
		$u$ be a non-negative lower semi-continuous weak solution of
		\begin{equation}\label{eq:l13}
			(-\Delta_p)^s u \ge 0\quad\text{in }\mathbb{R}^N.
		\end{equation} 
		
		We know, by Theorem \ref{Theorem:Fundamental},
		$v_\beta(x)=|x|^\beta$ is a weak solution of
		\begin{equation}\label{eq:l11}
				(-\Delta_p)^s v_\beta(x)= 
					\mathcal{C}(\beta)|x|^{\beta(p-1)-sp}\quad\text{in }
					\mathbb{R}^N\setminus\{0\}
		\end{equation}
		with $\mathcal{C}(\beta)>0$ (see \eqref{signodecbeta}). We now define 
		\[
			\theta^\varepsilon_\beta(x)\coloneqq m(r)
			\begin{cases}
				1&\text{if }0\le |x|<\varepsilon,\\
				\tfrac{R^\beta-|x|^\beta}{R^\beta-\varepsilon^\beta}
				&\text{if }\varepsilon\le |x|<R,\\
				0&\text{if }R\le |x|,
			\end{cases}
		\]
		where $m(r)\coloneqq \min\{u(x)\colon x\in \overline{B_r}(0)\}.$

		First, we prove the following auxiliary result.
	
		\begin{lemma}\label{lemma:aux12}
			For $\varepsilon$ sufficiently small, $\theta^\varepsilon_\beta$ is a weak solution of
			\begin{equation}\label{eq:l12}
				(-\Delta_p)^s \theta^\varepsilon_\beta(x)\le0\quad\text{in }A_{r,R}.
			\end{equation}
		\end{lemma}
		\begin{proof}
			Let $\varphi\in C^\infty_0(A_{r,R})$ be non-negative. Then
			\[
				\int\displaylimits_{\mathbb{R}^{2N}}
				\frac{\Psi_p
				(\theta^\varepsilon_\beta(x)
				-\theta^\varepsilon_\beta(y))(\varphi(x)-\varphi(y))}{|x-y|^{N+sp}}dxdy=I_1+I_2+I_3,
			\]
			where
			\begin{align*}
				I_1=&-\frac{2m(r)^{p-1}}{|R^\beta-\varepsilon^\beta|^{p-1}}
				\int\displaylimits_{A_{r,R}}\int\displaylimits_{B_\varepsilon(0)}
				\dfrac{||x|^\beta-\varepsilon^\beta|^{p-1}}{|x-y|^{N+sp}}\varphi(x) dydx\\
				I_2=&\frac{2m(r)^{p-1}}{|R^\beta-\varepsilon^\beta|^{p-1}}
				\int\displaylimits_{A_{r,R}}\int\displaylimits_{\mathbb{R}^N\setminus B_R(0)}
				\dfrac{|R^\beta-|x|^\beta|^{p-1}}{|x-y|^{N+sp}}\varphi(x) dydx\\
				I_3=&\frac{m(r)^{p-1}}{|R^\beta-\varepsilon^\beta|^{p-1}}
				\int\displaylimits_{A_{\varepsilon,R}^2}
				\frac{\Psi_p
				(|y|^\beta-|x|^\beta)(\varphi(x)-\varphi(y))}{|x-y|^{N+sp}}dxdy.
			\end{align*} 
			Observe that
			\[
				I_3 = \frac{m(r)^{p-1}}{|R^\beta-\varepsilon^\beta|^{p-1}}
				\int\displaylimits_{\mathbb{R}^{2N}}
				\frac{\Psi_p
				(|y|^\beta-|x|^\beta)(\varphi(x)-\varphi(y))}{|x-y|^{N+sp}}dxdy
				+J_1+J_2
			\]
			with
			\begin{align*}
				J_1&=\frac{2m(r)^{p-1}}{|R^\beta-\varepsilon^\beta|^{p-1}}
				\int\displaylimits_{A_{r,R}}\int\displaylimits_{B_\varepsilon(0)}
				\dfrac{||x|^\beta-|y|^\beta|^{p-1}}{|x-y|^{N+sp}}\varphi(x) dydx,\\
				J_2=&-\frac{2m(r)^{p-1}}{|R^\beta-\varepsilon^\beta|^{p-1}}
				\int\displaylimits_{A_{r,R}}\int\displaylimits_{\mathbb{R}^N\setminus B_R(0)}
				\dfrac{||y|^\beta-|x|^\beta|^{p-1}}{|x-y|^{N+sp}}\varphi(x) dydx.
			\end{align*}
			Then, using that $v_\beta(x)=|x|^\beta$ is a weak solution of \eqref{eq:l11}, we have that
			\begin{align*}
				&\int\displaylimits_{\mathbb{R}^{2N}}
				\frac{\Psi_p
				(\theta^\varepsilon_\beta(x)
				-\theta^\varepsilon_\beta(y))(\varphi(x)-\varphi(y))}{|x-y|^{N+sp}}dxdy=\\
				&=\frac{2m(r)^{p-1}}{|R^\beta-\varepsilon^\beta|^{p-1}}
				\int\displaylimits_{A_{r,R}}\int\displaylimits_{B_\varepsilon(0)}
				\dfrac{||x|^\beta-|y|^\beta|^{p-1}-||x|^\beta
				-\varepsilon^\beta|^{p-1}}{|x-y|^{N+sp}}\varphi(x) dydx\\
				&-\frac{2m(r)^{p-1}}{|R^\beta-\varepsilon^\beta|^{p-1}}
				\int\displaylimits_{A_{r,R}}\int\displaylimits_{\mathbb{R}^N\setminus B_R(0)}
				\dfrac{||y|^\beta-|x|^\beta|^{p-1}-|R^\beta-|x|^\beta|^{p-1}}{|x-y|^{N+sp}}\varphi(x) 
				dydx\\
				&-\mathcal{C}(\beta)\frac{m(r)^{p-1}}{|R^\beta-\varepsilon^\beta|^{p-1}}
				\int\displaylimits_{A_{r,R}}|x|^{\beta(p-1)-sp}\varphi(x) dx\\
				&\le\frac{2m(r)^{p-1}}{|R^\beta-\varepsilon^\beta|^{p-1}}
				\int\displaylimits_{A_{\varepsilon,R}}\int\displaylimits_{B_\varepsilon(0)}
				\dfrac{||x|^\beta-|y|^\beta|^{p-1}-||x|^\beta
				-\varepsilon^\beta|^{p-1}}{|x-y|^{N+sp}}\varphi(x) dydx\\
				&-\mathcal{C}(\beta)\frac{m(r)^{p-1}}{|R^\beta-\varepsilon^\beta|^{p-1}}
				\int\displaylimits_{A_{\varepsilon,R}}|x|^{\beta(p-1)-sp}\varphi(x) dx.
			\end{align*}
			Note that, for any $x\in A_{r,R},$ using the change of variable $z=\tfrac{y}{|x|},$ we have
			\begin{align*}
				\int\displaylimits_{B_\varepsilon(0)}&
				\dfrac{||x|^\beta-|y|^\beta|^{p-1}-||x|^\beta
				-\varepsilon^\beta|^{p-1}}{|x-y|^{N+sp}}dy=\\
				&=\int\displaylimits_{B_\varepsilon(0)}
				\dfrac{|1-\left(\tfrac{|y|}{|x|}\right)^\beta|^{p-1}-|1
				-\left(\tfrac{\varepsilon}{|x|}\right)^\beta|^{p-1}}{\left||\tfrac{x}{|x|}
					-\tfrac{y}{|x|}\right|^{N+sp}}dy|x|^{\beta(p-1)-N-sp}\\
				&=\int\displaylimits_{B_{\frac{\varepsilon}{|x|}}(0)}
				\dfrac{|1-|z|^\beta|^{p-1}-|1
				 -\left(\tfrac{\varepsilon}{|x|}\right)^\beta|^{p-1}}{\left|\tfrac{x}{|x|}
					-z\right|^{N+sp}}dz|x|^{\beta(p-1)-sp}\\
				&\le \frac{\alpha_N\varepsilon^N}
					{\left(1-\tfrac{\varepsilon}{R}\right)^{N+sp}}
				|x|^{\beta(p-1)-sp}.
			\end{align*}
			Therefore,
			\begin{align*}
				\int\displaylimits_{\mathbb{R}^{2N}}&
				\frac{\Psi_p
				(\theta^\varepsilon_\beta(x)
				-\theta^\varepsilon_\beta(y))(\varphi(x)-\varphi(y))}{|x-y|^{N+sp}}dxdy=\\
				&\le\frac{m(r)^{p-1}}{|R^\beta-\varepsilon^\beta|^{p-1}}
				\left(\frac{2\alpha_N\varepsilon^N}
				{\left(1-\tfrac{\varepsilon}{R}\right)^{N+sp}}-\mathcal{C}(\beta)\right)
				\int\displaylimits_{A_{\varepsilon,R}}|x|^{\beta(p-1)-sp}\varphi(x) dx,
				<0
			\end{align*}
			provided $\varepsilon$ is small enough.
		\end{proof}	
		
		We now prove our Liouville-type theorem for the case $sp>N.$
		\begin{theorem}\label{theorem:Liouville1} 
			Let $N\ge2,0<s<1,$ and $1<p<\infty.$ If $N<ps$ and $u$ is a non-negative
			lower semi-continuous weak solution of
			\eqref{eq:l13},
			then $u$ is constant.
		\end{theorem}
	
		\begin{proof}		
			Let $\beta\in\left(0,\tfrac{ps-N}{p-1}\right)$ and $0<\varepsilon<1<r<R.$  
			By Lemma \ref{lemma:aux12}, for $\varepsilon$ sufficiently small, $\theta^\varepsilon_\beta$ 
			is a weak solution of \eqref{eq:l12} and it is easy to see that 
			$\theta^\varepsilon_\beta\le u$ 
			in $\mathbb{R}^N\setminus A_{r,R}.$
			Thus, by the comparison principle, we have that
			$\theta^\varepsilon_\beta\le u$ in $\mathbb{R}^N.$ Therefore $m(r)\le u(x)$ for any $|x|\ge r.$ Then
			there is $x_0\in \overline{B_r}(0)$ such that $u(x_0)\le u(x)$ for any $x\in \mathbb{R}^N.$ 
			
			On the other hand, by \cite{MR4030247} 
			(see also \cite{MR4333435}), 
			we know the $u$ is a viscosity solution of \eqref{eq:l13}. 
			Finally, since $u$ attains its minimum, we can conclude that $u$ is constant.
		
		\end{proof}
	
	\subsection{Cases \texorpdfstring{$N=sp$}{N=sp}} 
	    Let $0<\varepsilon<1<r<R.$ In this case, we take 
	    a non-negative function $\zeta_\varepsilon\in 
	    C^\infty_c(\Omega)$ such that
	    \[
               \zeta_\varepsilon(x)= \begin{cases}
                    1 & \text{if } x\in B_{\tfrac{\varepsilon}{2}}(0),\\
                    0 &\text{if } x\in \mathbb{R}^N\setminus
                    B_{\varepsilon}(0),
                \end{cases}	    
	    \]
	    and
	    \[
	        \upxi_\varepsilon (x)\coloneqq
	        \begin{cases}
	            \log(|x|)-\log(\varepsilon)+\kappa
	            \zeta_\varepsilon(x)& \text{if } x\in 
	            B_{\varepsilon}(0),\\
	             0 &\text{if } x\in \mathbb{R}^N\setminus
                    B_{\varepsilon}(0),
	        \end{cases}
	    \]
	    where $\kappa$ is a positive constant to be chosen later.
	    
	    By Theorem \ref{Theorem:Fundamental} and
	    \cite[Lemma 2.8]{MR3593528}, we have that 
	    $\upvarrho_\varepsilon(x)\coloneqq \upxi_\varepsilon(x)-\log(|x|)$
	    satisfies
	    \[
                (-\Delta_p)^s\upvarrho_\varepsilon(x)=h(x)\quad\text{in } A_{r,R}	    
	    \]
	    where for a.e. Lebesgue point $x\in A_{r,R}$
	    \[
	        h(x)\coloneqq
	        2\int_{B_\varepsilon(0)}
	          \frac{F(x,y)}{|x-y|^{N+sp}}dy
	    \]
	    with
	    \[
                F(x,y)\coloneqq \Psi_p\left(-\log\left(\left|x\right|\right)+\log\left(\left|y\right|\right)-
                \upxi_\varepsilon\left(y\right)\right)-\Psi_p\left(-\log\left(\left|x\right|\right)+\log\left(\left|y\right|\right)\right).
	     \]
            
            Observe that, for a.e. Lebesgue point $x\in A_{r,R}$ and
            for any $y\in B_\varepsilon(0),$ we have
            \begin{align*}
                F(x,y)&=\Psi_p(-\log(|x|)+\log(\varepsilon)
                -\kappa \zeta_\varepsilon(y))+(\log(|x|)-\log(|y|)|^{p-1}\\
               &=(\log(|x|)-\log(|y|)|^{p-1}-(\log(|x|)-\log(\varepsilon)+
               \kappa \upxi_\varepsilon(y))^{p-1}.
            \end{align*}
            Then for a.e. Lebesgue point $x\in A_{r,R}$
            \[
                  h(x)= 2\int_{B_\varepsilon(0)}
                  \frac{(\log(|x|)-\log(|y|))^{p-1}-(\log(|x|)-\log(\varepsilon)+
               \kappa \upxi_\varepsilon(y))^{p-1}}{|x-y|^{N+sp}}dy.
             \]
    	    Now we choose $\kappa$ large enough so that $h(x)\le 0$  
    	    for a.e. Lebesgue point $x\in A_{r,R}.$ Then, we can prove the second 
    	    case of our Liouville-type theorem.
            \begin{theorem}\label{theorem:Liouville2} 
	        Let $N\ge2,0<s<1,$ and $1<p<\infty.$ If $N=ps$ and $u$ is a 
	        non-negative lower semi-continuous weak solution of \eqref{eq:l13},
		then $u$ is constant.
            \end{theorem}
	    \begin{proof}
                Let $u$ be a 
	        non-negative lower semi-continuous weak solution of \eqref{eq:l13}
	        and
	         \[
                    \uptheta_\varepsilon(x)=m(r)
                    \dfrac{\upvarrho_\varepsilon(x)+\log(R)}{-\log(\varepsilon)+
                    \kappa+\log(R)}    	    
    	        \]
    	        	where $m(r)\coloneqq \min\{u(x)\colon x\in \overline{B_r}(0)\}.$
    	        	Then $\uptheta_\varepsilon$ is a weak solution of \eqref{eq:l12}
    	        	and it is easy to see that $\uptheta_\varepsilon\le u$ 
    	        	in $\mathbb{R}^N\setminus A_{r,R}.$
		Thus, by the comparison principle, we have that
		$\uptheta_\varepsilon\le u$ in $\mathbb{R}^N.$ 
		Therefore $m(r)\le u(x)$ for any $|x|\ge r.$ Then
		there is $x_0\in \overline{B_r}(0)$ such that $u(x_0)\le u(x)$ for 
		any $x\in \mathbb{R}^N.$ 
			
		On the other hand, by \cite{MR4030247} (see also \cite{MR4333435}), 
		we know that $u$ is a viscosity solution of \eqref{eq:l13}. 
		Finally, since $u$ attains its minimum, we can conclude that $u$ is 
		constant.
            \end{proof}
\section{A nonlinear Liouville-type theorem}\label{st2}
    In this last section, we prove our non-linear Liouville-type theorem
    (see Theorem \ref{theorem:nonlinear}). As before, we split the proof in
    two cases.
    
    \subsection{The sub-critical case} First, we show the following result.
    	\begin{theorem}
	\label{theorem:nonlinear1}
		Let $N\ge2,0<s<1,$  $1<p<\infty,$ and $N>ps.$ 
		If $0<q< \tfrac{N(p-1)}{N-ps}$ and $u\in C(\mathbb{R}^N)$ is a non-negative weak solution of 
			\begin{equation}\label{eq:nl11}
				(-\Delta_p)^s u-u^q \ge 0\quad\text{in }\mathbb{R}^N.
			\end{equation}
			then $u\equiv0.$
		 
    \end{theorem}
    \begin{proof}
        Let $u$ be a non-negative lower semi-continuous weak solution of 
        \eqref{eq:nl11}. {By \cite{MR4333435}, $u$ is a viscosity solution 
        of \eqref{eq:nl11}.}
        
        We suppose by contradiction that $u\not\equiv 0$
        in $\mathbb{R}^N.$ By \cite[Theorem 1.2]{MR3631323}, we have that
        $u>0$ a.e. in $\mathbb{R}^N.$

        On the other hand, by \cite{MR4030247} (see also \cite{MR4333435}), 
	we know the $u$ is a viscosity solution of 
	\[
            (-\Delta_p)^s u(x)\ge 0\quad\text{ in }\mathbb{R}^N.	
	\]
	Therefore $u>0$ in $\mathbb{R}^N.$

        Let's take a function $\mu\in C^\infty([0,\infty),\mathbb{R})$
        such that $\mu$ is non-increasing, $0\le\mu\le1,$ and
        \[
            \mu(t)=
            \begin{cases}
               1 &\text{if } 0\le t\le \tfrac12,\\
               0  &\text{if } t\ge 1.\\    
            \end{cases}        
        \]  
        
        Then, by \cite[Proposition 2.12]{MR3593528}, there is a positive 
        constant $C$ such that $w(x)=\mu(|x|)$ satisfies in strongly sense
        \[
            (-\Delta_p)^s w(x)\le C \quad\text{ in } B_1(0).       
        \]
        
        Now, for any $R>1$, 
        we take 
        \[
            \upzeta(x)= m\left(\tfrac{R}2\right)\mu\left(\tfrac{|x|}R\right),
        \]       
        	where $m\left(\tfrac{R}2\right)= 
        	\min\left\{u(x)\colon x\in \overline{B_{\tfrac{R}2}}(0)\right\}.$
        	Observe that $\upzeta$ satisfies in strongly sense
        \begin{equation}
            \label{eq:nonlinear1}
            (-\Delta_p)^s \upzeta(x)\le m\left(\tfrac{R}2\right)^{p-1} 
            \frac{C}{R^{ps}} \quad\text{ in } B_R(0).            
        \end{equation}
        Here, $C$ is a positive constant independent of $R.$
        	
        	On the other hand, since $\upzeta(x)\le u(x)$ in $\mathbb{R}^N\setminus A_{\frac{R}2,R}$ and $u$ is lower semi-continuous function, there is $x_R\in B_R(0)$
        	such that $u(x_R)-\upzeta(x_R)\le u(x)-\upzeta(x)$ for any $x\in \mathbb{R}^N.$
        	
        	We divide the rest of the proof into two cases.
        	
        	\medskip

        \noindent	{\it Case 1:} $\tfrac{R}{2}<|x_R|<R.$ 
        	We take $r\ll\mathrm{dist}(x_R,\partial B_R(0)),$ and
        	\[
        	    \phi_r(x)=
        	        \begin{cases}
                    \upzeta(x)- \upzeta(x_R)+u(x_R) &\text{ if } x\in B_r(x_R),\\
                    u(x)  &\text{ if } x\in \mathbb{R}^N\setminus B_r(x_R).\\  
                \end{cases}
        	\]
        	Note that $\phi_r(x)\in C^\infty(B_r(x_R)),$ $\phi_r(x_R)=u(x_R),$ 
        	$\phi_r(x)\le u(x)$  in $B_r(x_R)$ and 
        	$\nabla\phi_r(x_R)=\nabla\upzeta_r(x_R)\neq 0.$ Then,
        	since $u$ is a viscosity solution of \eqref{eq:nl11}, we have that
        	\[
            u(x_R)^q\le(-\Delta_p)^s\phi_r(x_R).        	
        	\]
        	Now, using that $u(x_R)-\upzeta(x_R)\le u(x)-\upzeta(x)$ for any 
        	$x\in \mathbb{R}^N,$ $x_R\in B_R(0),$ and \eqref{eq:nonlinear1}, we get
        \[
            m(R)^q\le (-\Delta_p)^s\phi_r(x_R)\le(-\Delta_p)^s\upzeta_r(x_R)\le
            m\left(\tfrac{R}2\right)^{p-1} 
            \frac{C}{R^{ps}}.
        \]
        Then, by Lemma \ref{lemma:Had2}, we have
        \[
            m(R)^{q}\le \frac{C}{R^{ps}}m(R)^{p-1}\quad\forall R>1
        \]        	
       where $C$ is a positive constant independent of $R.$
       If $0<q\le p-1,$ we obtain a contradiction. On the other hand, if 
       $q>p-1,$ we have
       \[
            m(R)\le CR^\kappa       
       \]
       where $\kappa=-\tfrac{ps}{q-p+1}.$ Since $p-1<q<\tfrac{N(p-1)}{N-ps},$ there is
       $\beta\in (\tfrac{N}{p-1},\tfrac{ps-N}{p-1})$ such that $\kappa<\beta.$ Then,
       by Lemma \ref{eq:Had1_1}, there is $r_0>1$ and a positive constant such that
         \[
            m(R)\le CR^\kappa\le C R^{\kappa-\beta} m(R)\quad \forall R>r_0.      
       \]
       We again obtain a contradiction.
       
       \medskip

        \noindent{\it Case 2:} $|x_R|\le \tfrac{R}{2}.$ 
        Then
        \[
            0\le u(x_R)-m(\tfrac{R}2)=u(x_R)-\upzeta(x_R)\le u(x)-\upzeta(x)
            \quad\forall x\in\mathbb{R}^N.        
        \]
        In particular, if we $\tilde{x}\in B_{\frac{R}2}(0)$ so that 
        $u(\tilde{x})=m(\tfrac{R}2)$ we have that 
        \[
            0\le u(x_R)-m(\tfrac{R}2)\le 0.
        \]
        Therefore $u(x_R)=m(\tfrac{R}2)$ and $\upzeta(x)\le u(x)$ for any 
        $x\in\mathbb{R}^N.$
        
        Thus if $p>\tfrac{2}{2-s},$ we can proceed as in the previous case.
        But if $1<p\le\tfrac2{2-s},$ we have a problem because $x_R$ is a 
        	critical point of $\upzeta$ but it is not 
        	isolated. 
        	Then $\phi_r$ is not an admissible test function.
        	To solve this problem, we take
        	\[
        	    \tilde{\phi}_r (x)=
             \begin{cases}
                	   \upzeta(x)- m\left(\tfrac{R}2\right) |x-x_R|^\gamma &\text{if } x\in B_{r}(x_R),\\
                    u(x) &\text{if } x\in B_{r}(x_R),\\
                \end{cases}      	            	     
        	\]
         as a test function with
         \[
              \gamma  >\tfrac{sp}{p-1} \quad \text{ and }\quad
              r<R^{-\tfrac{sp}{\gamma(p-1)-sp}}.
         \]  
         
         Observe that
         \begin{align*} 
            m(R)^q\le& u(x_R)^q\le (-\Delta_p)^s\tilde{\phi}_r(x_R)\\
                \le& \int_{B_r(x_R) }\dfrac{|\upzeta(x_R)-\upzeta(x)+
                 m\left(\tfrac{R}2\right)|x-x_R|^\gamma|^{p-1}}
                {|x-x_R|^{N+sp}}dx\\
                &+ \quad \int_{\mathbb{R}^N\setminus B_r(x_R) }
                \dfrac{|u(x_R)-u(x)|^{p-2}(u(x_R)-u(x))}
                {|x-x_R|^{N+sp}}dx.      
        \end{align*} 
        Note that, $0<p-1\le\tfrac{s}{2-s}<1$ and $\upzeta(x_R)-\upzeta(x)\ge 0$ for
        any $x\in\mathbb{R}^N.$ Then, using that 
        \[
            (a+b)^{q}\le a^{q}+b^q \quad\forall \, a,b\ge0 \quad q\in(0,1]        
        \]
        (see \cite{MR3593528}), $\upzeta(x)\le u(x)$ for any 
        $x\in\mathbb{R}^N,$ and \eqref{eq:nonlinear1}, we have that
         \begin{align*} 
            m(R)^q\le& u(x_R)^q\le (-\Delta_p)^s\tilde{\phi}_r(x_R)\\
                \le& m\left(\tfrac{R}2\right)^{p-1} 
                 \int_{B_r(x_R) }\dfrac{|x-x_R|^{\gamma(p-1)}}
                {|x-x_R|^{N+sp}}dx+ (-\Delta_p)^s\upzeta(x_R)\\
                \le& C\left(r^{\gamma(p-1)-sp}+ \tfrac{1}{R^{ps}}\right)
                 m\left(\tfrac{R}2\right)^{p-1}\\
                 \le &\frac{C}{R^{ps}} m\left(\tfrac{R}2\right)^{p-1}\quad\forall R>1,
        \end{align*} 
        where $C$ is a positive constant independent of $R.$
        
        Now the proof follows exactly the proof of Case 1.            	
    \end{proof}

    \subsection{The super-critical} To conclude this article, we prove the following
    result.
    	\begin{theorem}
	\label{theorem:nonlinear3}
		Let $N\ge2,0<s<1,$ $1<p<\infty,$ and $N>ps.$ 
		If $q>\tfrac{N(p-1)}{N-ps}$ then there is a positive solution of \eqref{eq:nl11}.
	\end{theorem}
	
        \begin{proof}
            In this case, we take $\upkappa=\tfrac{sp}{q-p+1}$ and
            \[
                w(x)=\dfrac{1}{(1+|x|)^{\upkappa}}.            
            \]
             Observe that, since $q>\tfrac{N(p-1)}{N-ps}$ we have that 
            $0<\upkappa<\tfrac{N-ps}{p-1}.$ 
            
            For any $x\in\mathbb{R}^N\setminus\{0\},$
            we have that
            \begin{align*}
               2 \int_{\mathbb{R}^N}&
                \dfrac{\Psi_p(w(x)-w(y))}{|x-y|^{N+sp}} dy=2\int_{\mathbb{R}^N}
                \dfrac{\Psi_p
                \left( \tfrac{1}{(1+|x|)^{\upkappa}}-
                \tfrac{1}{(1+|y|)^{\upkappa}}\right)}{|x-y|^{N+sp}}
                dy\\
                &=2\dfrac{1}{(1+|x|^2)^{\upkappa(p-1)+sp+N}}
                \int_{\mathbb{R}^N}
                \dfrac{
                    \Psi_p
                \left(1-
                        \left(
                        \tfrac{1+|x|}{1+|y|}
                        \right)^{\upkappa} \right)
                }{\left|\tfrac{x}{1+|x|}-\tfrac{y}{1+|x|}\right|^{N+sp}}
                dy.
            \end{align*}
            Applying a rotation, we may assume that 
            $\tfrac{x}{|x|}=e_1=(1,0,\dots,0)\in\mathbb{R}^N,$ and via
            the change of variable 
            $z=\tfrac{1}{1+|x|}\left(y+\tfrac{x}{|x|}\right),$ and since
            \[
                \frac{x}{1+|x|}-\frac{y}{1+|x|}=
                \frac{x}{|x|}-\frac{x}{(1+|x|)|x|} -\frac{y}{1+|x|}=  
                \frac{x}{|x|}-\frac{1}{1+|x|}\left(y+\frac{x}{|x|}\right),           
            \]
            and $\upkappa(p-1)+sp=\upkappa q,$ we get
            \[2(1+|x|)^{\upkappa q}\int_{\mathbb{R}^N}
                \dfrac{\Psi_p(w(x)-w(y))}{|x-y|^{N+sp}} dy=
                2\int_{\mathbb{R}^N}
                \dfrac{
                    \Psi_p
                \left(1-
                        \left(
                        \frac{1+|x|}{1+|(1+|x|)z-e_1|}
                        \right)^{\upkappa} \right)
                }{|e_1-z|^{N+sp}}
                dz.
            \]
            Now, using that 
            \[
                (1+|x|)|z|\le 1+|(1+|x|)z-e_1|
            \]
            we have
            \[
               \int_{\mathbb{R}^N}
                \dfrac{
                   \Psi_p
                \left(1-
                        \left(
                        \tfrac{1+|x|}{1+|(1+|x|)z-e_1|}
                        \right)^{\upkappa} \right)
                }{|e_1-z|^{N+sp}}
                dz\ge\\
                \int_{\mathbb{R}^N}
                \dfrac{
                    \Psi_p
                \left(1-
                        \tfrac{1}{|z|^{\upkappa}}\right)
                }{|e_1-z|^{N+sp}}
                dz=\mathcal{C}(-\upkappa).
            \]
            Since, $0>-\upkappa>\tfrac{ps-N}{p-1},$ by \eqref{signodecbeta},
               we have that $\mathcal{C}(-\upkappa)>0.$ See also Remark \ref{remark.visco.weak}. Then,
               \[
                    \mathcal{C}(-\upkappa)^{\frac1{q-p+1}}w ,              
               \]
                 is a positive solution of \eqref{eq:nl11}.
        \end{proof}


\section*{Acknowledgments}
	L.M.D.P. was partially supported by CONICET grant PIP GI No. 11220150100036CO (Argentina), 
	Agencia grant PICT-2018-03183 (Argentina), and UBACyT grant 20020160100155BA (Argentina).
	
	A.Q. was partially supported by Fondecyt grant No. 1231585 (Chile). 
	
	Both L.M.D.P. and A.Q. were also partially supported by the Math AmSud program under project 23-MATH-08-MiLNE.
	
\def\cprime{$'$}

\end{document}